\documentclass[11pt,a4paper]{article}

\usepackage{graphicx}
\usepackage{amsmath}
\usepackage{amsfonts}
\usepackage{amssymb}
\usepackage{enumerate}
\usepackage{lscape}
\usepackage{longtable}
\usepackage{rotating}
\usepackage{color}
\usepackage{url}
\usepackage{rotating}
\usepackage{algpseudocode}
\usepackage[ruled]{algorithm}
\usepackage{pgf,tikz}
\usepackage{subfig}
\usepackage[titletoc]{appendix}

\usepackage{mathrsfs}
\usepackage{pgfplots}
\usetikzlibrary{arrows}
\usepackage{color}

\usepackage{bbm}

\usepackage{booktabs}
\usepackage{multirow}
\usepackage{mathtools}
\usepackage{makecell}

\usepackage[sort,numbers]{natbib}
\usepackage{eqparbox}%

\numberwithin{equation}{section}%

\usepackage[symbol]{footmisc}%

\newtheorem{proposition}{Proposition}[section]

\newtheorem{definition}{Definition}

\newtheorem{remark}{Remark}[section]

\newtheorem{assumption}{Assumption}[section]

\newenvironment{proof}[1][Proof]{\noindent \textbf{#1.} }{\hfill$\Box$\par\medskip}

\DeclareMathOperator*{\argmax}{argmax}
\DeclareMathOperator*{\argmin}{argmin}

\DeclareMathOperator{\conv}{conv}

\DeclareMathOperator{\OPT}{OPT}
\DeclareMathOperator{\RN}{RN}
\DeclareMathOperator{\RA}{RA}

\DeclareMathOperator{\minquotes}{``min''}

\newcommand{\beqn}[1]{\begin{equation}\label{#1}}
\newcommand{\eeqn}{\end{equation}}

\definecolor{darkgreen}{rgb}{0,0.6,0}
\definecolor{aau2}{rgb}{0.0, 0.5, 0.69}
\definecolor{aau3}{rgb}{0.0, 0.53, 0.74}
\definecolor{aau4}{rgb}{0.0, 0.48, 0.65}
\definecolor{aau5}{rgb}{0.0, 0.45, 0.73}
\definecolor{rsap}{RGB}{130, 36, 51}
\definecolor{gsap}{RGB}{112, 164, 137}
\definecolor{tud}{rgb}{0.43,0.73,0.11}
\definecolor{verde}{rgb}{0.33,0.53,0.11}

\definecolor{ttffqq}{rgb}{0.0, 0.48, 0.65} %{rgb}{0.43,0.73,0.11}
\definecolor{ffqqqq}{rgb}{0.0, 0.5, 0.69} %{rgb}{1,0,0}

\usetikzlibrary{arrows}

\tikzstyle{decision} = [diamond, draw, fill=blue!20,
text width=4.5em, text badly centered, node distance=3cm, inner sep=0pt]
\tikzstyle{block} = [rectangle, draw, fill=blue!20,
text centered, rounded corners, minimum height=4em]
\tikzstyle{line} = [draw, -latex']
\tikzstyle{cloud} = [draw, ellipse,fill=red!20, node distance=3cm,
minimum height=2em]
\tikzstyle{cloud2} = [draw, ellipse,fill=green!20, node distance=3cm,
minimum height=2em]
\begin{document}
	
	\title{Bilevel optimization with a multi-objective lower-level problem: Risk-neutral and risk-averse formulations}
	
	\author{
		T. Giovannelli\thanks{Department of Industrial and Systems Engineering, Lehigh University, Bethlehem, PA 18015-1582, USA ({\tt tog220@lehigh.edu}).}
		\and
		G. D. Kent\thanks{Department of Industrial and Systems Engineering, Lehigh University, Bethlehem, PA 18015-1582, USA ({\tt gdk220@lehigh.edu}).}
		\and
		L. N. Vicente\thanks{Department of Industrial and Systems Engineering, Lehigh University, Bethlehem, PA 18015-1582, USA ({\tt lnv@lehigh.edu}).}
	}
	
	\maketitle
	
	\begin{abstract}
	In this work, we propose different formulations and gradient-based algorithms for deterministic and stochastic bilevel problems with conflicting objectives in the lower level. Such problems have received little attention in the deterministic case and have never been studied from a stochastic approximation viewpoint despite the recent advances in stochastic methods for single-level, bilevel, and multi-objective optimization.

    To solve bilevel problems with a multi-objective lower level, different approaches can be considered depending on the interpretation of the lower-level optimality. An optimistic formulation that was previously introduced for the deterministic case consists of minimizing the upper-level function over all non-dominated lower-level solutions. In this paper, we develop new risk-neutral and risk-averse formulations, address their main computational challenges, and develop the corresponding deterministic and stochastic gradient-based algorithms.   
    \end{abstract}

%%%%%%%%%%%%%%%%%%%%%%%%%%%%%%%%%%%%%%%%%%%%%%%%%%%%%%%%%%%%%%%%%%%%%%%%%%%%%%%%%%%%%%%%%
\section{Introduction}
%%%%%%%%%%%%%%%%%%%%%%%%%%%%%%%%%%%%%%%%%%%%%%%%%%%%%%%%%%%%%%%%%%%%%%%%%%%%%%%%%%%%%%%%%

In bilevel multi-objective optimization (BMO), at least one of the two levels of the bilevel problem has multiple objectives and can be modeled using the formulation
\begin{equation}\label{prob:bilevel_multiobj}
	\begin{split}
	\min_{x \in \mathbb{R}^n, \, y \in \mathbb{R}^m} ~~ & F_u(x,y) \\
	\mbox{s.t.}~~ & x \in X \\
 & y \in \argmin_{y \in Y(x)} ~~ F_\ell(x,y).\\
	\end{split}
\end{equation}
The upper-level (UL) and lower-level (LL) objective functions $F_u: \mathbb{R}^n\times\mathbb{R}^m \to \mathbb{R}^p$ and $F_{\ell}: \mathbb{R}^n\times\mathbb{R}^m \to \mathbb{R}^q$ are given by~$F_u = (f_u^1, \ldots, f_u^p)$ and~$F_{\ell} = (f_{\ell}^1, \ldots, f_{\ell}^q)$, with~$f_u^i$ and $f_{\ell}^j$ real-valued functions for all~$i \in \{1,\ldots,p\}$ and~$j \in \{1,\ldots,q\}$, respectively, and~$p \ge 1$ or~$q \ge 1$. In this general formulation, the UL variables $x$ are subjected to UL constraints ($x\in X$) and the LL variables $y$ are subjected to LL constraints ($y\in Y(x)$); however, for the rest of this paper, it is assumed that the LL variables are unconstrained (i.e., $Y(x)=\mathbb{R}^m$).
The set~$X$ will be assumed closed and convex, which allows us to obtain feasible points with respect to the~UL constraints by taking orthogonal projections onto~$X$.
In general, when at least one of the two levels is a multi-objective optimization problem, we can use the name bilevel multi-objective optimization~\cite{GEichfelder_2010,GEichfelder_2020,ASinha_KDeb_2009}. When a multi-objective problem only arises in one of the two levels, one also finds the terminology semivectorial bilevel optimization~\cite{SDempe_PMehlitz_2020}. 
 
Bilevel multi-objective optimization problems arise in applications related to defense, renewable energy systems, and fair machine learning. In the defense sector, a multi-objective bilevel problem for facility location was proposed in~\cite{AMLessin_BJLunday_RRHill_2019} to prevent an adversary from entering a territory by relocating wireless sensors in order to maximize the exposure of the attacker to the sensors (single-objective~UL problem) and minimize the conflicting objectives given by the sensor relocation time and the total number of sensors (multi-objective LL problem). In~\cite{XLuo_YLiu_XLiu_2021}, the authors propose a bilevel formulation to minimize the environmental impact of renewable energy systems and the resulting government expenditure (multi-objective~UL problem) and the energy cost paid by the end users (single-objective~LL problem). Bilevel optimization has recently been adopted to solve fair machine learning (ML) problems~\cite{MSOzdayi_MKantarcioglu_RIyer_2021,YRoh_etal_2020,MMKamani_etal_2020}, where the goal is to minimize the prediction error on a validation dataset by training an ML model to avoid discriminatory predictions against people with sensitive attributes. To ensure accurate and fair prediction outcomes in real-life decision-making applications, accuracy and fairness loss functions must be jointly considered, thus leading to bilevel problems where both~UL and~LL problems are multi-objective.

In this paper, we consider bilevel multi-objective optimization problems with a multi-objective lower level (BMOLL), which can be obtained from problem~\eqref{prob:bilevel_multiobj} by considering~$p = 1$ and~$q > 1$.  Hence, the problem to solve is
\begin{equation}\label{prob:single_bilevel}
	\begin{split}
	\min_{x \in \mathbb{R}^n, \, y \in \mathbb{R}^m} ~~ & f_u(x,y) \\
	\mbox{s.t.}~~ & x \in X \\
 & y \in \argmin_{y \in \mathbb{R}^m} ~~ F_\ell(x,y),\\
	\end{split}
\end{equation}
where the~UL objective function~$f_u: \mathbb{R}^n\times\mathbb{R}^m \to \mathbb{R}$ is real-valued. Throughout the paper, we assume~$f_u$ to be continuously differentiable and all the functions~$f_{\ell}^j$, with~$j \in \{1,\ldots,q\}$, to be twice continuously differentiable. We will formalize these assumptions in Section~\ref{sec:LL_multiobj_prob}. Note that denoting~$\Psi: \mathbb{R}^n \rightrightarrows \mathbb{R}^m$ as the set-valued mapping representing the optimal solutions of the~LL problem (we will identify this as the set of weak LL Pareto minimizers in Section~\ref{sec:definitions}), one can reformulate~\eqref{prob:single_bilevel} as follows:
\begin{equation}\label{eq:single_multi_bilevel}
    \underset{x \in \mathbb{R}^n, \, y \in \mathbb{R}^m}{\minquotes} ~~ \{f_u(x,y)~|~x \in X, \ y \in \Psi(x)\},
\end{equation}
where the notation~``min'' with double quotes is used in the literature to denote the ambiguity arising when there are multiple~LL optimal solutions~\cite{ASinha_PMalo_KDeb_2018,SDempe_PMehlitz_2020}. Since the~LL problem has conflicting objectives, given $x\in X$, a single optimal solution must be determined among the set~$\Psi(x)$, and several criteria can be considered that potentially lead to different solutions of the bilevel problem. In the deterministic case, researchers have focused on \textit{optimistic} formulations, excluding all those cases where the solution determined among the set of~LL Pareto optimal points is not the most favorable for the~UL problem.   

As opposed to bilevel problems with~UL and~LL single-objective functions~\cite{NCouellan_WWang_2015,NCouellan_WWang_2016,SGhadimi_MWang_2018,MHong_etal_2020,TChen_YSun_WYin_2021_closingGap,TGiovannelli_GKent_LNVicente_2022,DSow_KJi_YLiang_2021} (see also~\cite{RLiu_JGao_etal_2021,CChen_etal_2022,TGiovannelli_GKent_LNVicente_2022} for recent reviews), gradient-based methods for bilevel multi-objective problems have received less attention in the literature. Deterministic approaches for bilevel optimization with multi-objective~LL problems have been proposed in~\cite{YLv_ZWan_2014a,YLv_ZWan_2014b,RAndreani_etal_2019,SDempe_PMehlitz_2020}, where the LL problem is transformed into a single-objective problem by weighting the LL objective functions according to the weighted-sum approach that is utilized in multi-objective optimization~\cite{MEhrgott_2005}. The weights are then included among the UL optimization variables. 
%In particular,~\cite{YLv_ZWan_2014a,YLv_ZWan_2014b} propose to replace the new single-objective~LL problem with its optimality conditions. Conversely,~\cite{RAndreani_etal_2019} deals with a constrained~LL problem and proposes an inexact feasibility restoration strategy to restore feasible iterates once they leave the feasible region. 
Problems with multi-objective~UL and~LL problems ($p,q>1$) have been addressed in~\cite{XShi_HSXia_2001}, which uses the~$\varepsilon$-constraint method~\cite{MEhrgott_2005} at both levels to obtain a single-objective bilevel problem.
%and then replaces the resulting constrained~LL problem by its optimality conditions. 
To the best of our knowledge, the only stochastic gradient-based algorithm for bilevel multi-objective problems has been proposed in~\cite{AGu_etal_2022}, where multiple objectives are considered at the upper level. We point out that the problem solved in~\cite{AGu_etal_2022} is significantly different from problem~\eqref{prob:bilevel_multiobj} since in~\cite{AGu_etal_2022} there are multiple~LL problems and each~UL objective function only depends on the optimal solution of one~LL problem. Moreover, in~\cite{AGu_etal_2022}, the authors do not attempt to determine the~UL Pareto front and consider a robust formulation of the~UL problem to minimize the maximum optimal value among all the~UL objective functions. 
% stochastic gradient-based methods for bilevel multi-objective problems are still unexplored.

As an alternative to the known optimistic formulation developed for~(\ref{prob:single_bilevel}), we propose new \textit{risk-neutral} and \textit{risk-averse} formulations, address their computational challenges, and develop their corresponding deterministic and stochastic gradient descent algorithms. Both formulations are inspired by looking at~$y$ as a parameter, rather than as a variable. In the risk-neutral case, we minimize a new function describing the mean of the UL function over an $x$-dependent set of optimal LL solutions. We propose a formulation that is tractable (by rather taking the mean over LL weights) and can lead to efficient algorithms (by sampling the weights).
The risk-averse case requires using the extension of Danskin's Theorem to the case where the maximum is taken over an $x$-dependent efficient solution set.

This paper is organized as follows.
We first review the bilevel stochastic gradient method in Section~\ref{sec:multilevel_multiobj_alg} (for a single objective in both the UL and LL).
In Section~\ref{sec:definitions}, we introduce basic definitions, results, and general assumptions.
We describe the known optimistic formulation for~(\ref{prob:single_bilevel}) in Section~\ref{sec:LL_multiobj_prob}. The new risk-neutral and risk-averse formulations for~(\ref{prob:single_bilevel}), as well as the corresponding gradient-based algorithms, are introduced in Sections~\ref{sec:adj_grad_LL_multiobj_prob} and~\ref{sec:adj_grad_LL_multiobj_prob_general}, respectively. Numerical results for synthetic bilevel problems with a multi-objective lower level are reported in Section~\ref{sec:numerical_exp}, which also describes the practical implementations of the proposed methods. Finally, we draw some concluding remarks in Section~\ref{sec:UL_multiobj_prob}, in particular how to develop the cases $p>1,q=1$ and $p,q>1$ from known building blocks.

%%%%%%%%%%%%%%%%%%%%%%%%%%%%%%%%%%%%%%%%%%%%%%%%%%%%%%%%%%%%%%%%%%%%%%%%%%%%%%%%%%%%%%%%%
\section{A review of bilevel stochastic gradient methods}\label{sec:multilevel_multiobj_alg}
%%%%%%%%%%%%%%%%%%%%%%%%%%%%%%%%%%%%%%%%%%%%%%%%%%%%%%%%%%%%%%%%%%%%%%%%%%%%%%%%%%%%%%%%%

Bilevel optimization with multi-objective upper or lower levels is based on (or uses as a reference) the bilevel single-objective~(BO) case ($p=q=1$), which can be modeled using the formulation
	\begin{equation}\label{prob:bilevel}
	\begin{split}
	\min_{x \in \mathbb{R}^n, \, y \in \mathbb{R}^m} ~~ & f_u(x,y) \\
	\mbox{s.t.}~~ & x \in X \\
 & y \in \argmin_{y \in \mathbb{R}^m} ~~ f_\ell(x,y).\\
	\end{split}
	\end{equation}
In~BO problems, the goal of the~UL problem is to minimize the~UL objective function~$f_u: \mathbb{R}^n\times\mathbb{R}^m \to \mathbb{R}$ over the~UL variables~$x$, which are subjected to~UL constraints ($x \in X$), and~LL variables~$y$, which are subjected to being an optimal solution of the~LL problem. The goal of the~LL problem is to minimize the~LL objective function $f_\ell: \mathbb{R}^n\times\mathbb{R}^m \to \mathbb{R}$ over the LL variables~$y$. 

Assuming that there exists a unique solution~$y(x)$ to the~LL problem for all $x \in X$, problem~BO is equivalent to a problem posed solely in terms of the~UL variables~$x$, and is given by
\begin{equation} \label{reduced}
\min_{x \in \mathbb{R}^n} \; f(x)=f_u(x,y(x)) \quad
	\mbox{s.t.} \quad x \in X.
\end{equation}

Recalling the assumptions that~$f_u$ is continuously differentiable,~$f_{\ell}$ twice continuously differentiable, and further assuming~$\nabla^2_{yy}f_\ell(x,y(x))$ to be non-singular, 
the gradient of~$f$ at~$x$ can be obtained from the well-known adjoint (or hypergradient) formula
\begin{equation} \label{adjoint}
\nabla f \; = \; \nabla_x f_u - \nabla_{xy}^2 f_\ell (\nabla^2_{yy}f_\ell)^{-1} \nabla_y f_u,
\end{equation}
where all gradients and Hessians on the right-hand side are evaluated at~$(x,y(x))$. The steepest descent direction for~$f$ at~$x$ is denoted by~$d(x, y(x)) = - \nabla f(x)$. To obtain the adjoint formula, one can apply the chain rule to $f_u(x,y(x))$, which leads to
$\nabla f = \nabla_x f_u + \nabla y \nabla_y f_u$.
Then, the Jacobian~$\nabla y(x)\in\mathbb{R}^{n\times m}$ can be derived from the~LL first-order necessary optimality conditions~$\nabla_y f_\ell(x,y(x))=0$ by applying the chain rule to both sides of this equation with respect to $x$. The~LL optimal solution function~$y(\cdot)$ is continuously differentiable due to the implicit function theorem~\cite{Ruding_1953}. The equation resulting from the application of the chain rule is given by~$\nabla^2_{yx}f_\ell + \nabla^2_{yy}f_\ell\nabla y^{\top} \; = \; 0$ (where all gradients and Hessians are evaluated at~$(x,y(x))$), which leads to
$\nabla y = - \nabla^2_{xy}f_\ell (\nabla^2_{yy}f_\ell)^{-1}$.

In Algorithm~\ref{alg:MOBSG}, we report a general framework for the bilevel stochastic gradient (BSG) \linebreak method~\cite{TGiovannelli_GKent_LNVicente_2022} for stochastic BO problems. Such a framework will be adapted to develop different algorithms for the~BMOLL problem considered in this paper. We adopt~$\xi_k$ to denote the random variables used to obtain stochastic estimates for~UL and~LL gradients and Hessians. 
An initial point $(x_0,y_0)$ and a sequence of positive scalars~$\{\alpha_k\}$ are required as input. In Step~1, any appropriate optimization method can be applied to obtain an approximate LL solution $\tilde y$ by solving the LL problem to a specified degree of accuracy. In Step~2, one computes an approximate negative BSG~$d(x_k, \tilde y_k, \xi_k)$, defined by the adjoint gradient in~\eqref{adjoint}, to update the~UL variables. In Step~3, the vector $x$ is updated by choosing a step size taken from the sequence of positive scalars~$\{\alpha_k\}$. When $X$ is a closed and convex constrained set different from $\mathbb{R}^n$, an orthogonal projection of $x_{k} + \alpha_k \, d(x_k, \tilde y_k, \xi_k)$ onto~$X$ is required (such a projection can be computed by solving a convex optimization problem). Regarding the stepsize sequence~$\{\alpha_k\}$, popular options in the stochastic gradient literature are fixed or decaying stepsize sequences~\cite{LBottou_FECurtis_JNocedal_2018}.
\begin{algorithm}[H]
	\caption{Bilevel Stochastic Gradient (BSG) Method}\label{alg:MOBSG}
	\begin{algorithmic}[1]
		\medskip
		\item[] {\bf Input:} $(x_0,y_0) \in \mathbb{R}^n \times \mathbb{R}^m$, $\{\alpha_k\}_{k \geq 0} > 0$.
		\medskip
		\item[] {\bf For $k = 0, 1, 2, \ldots$ \bf do}
		\item[] \quad\quad {\bf Step 1.}
		 Obtain an approximation $\tilde{y}_k$ to the LL optimal solution $y(x_k)$. 
 \nonumber
    \item[] \quad\quad {\bf Step 2.}
		 Compute a negative BSG $d(x_k,\tilde{y}_k,\xi_k)$.
		\item[] \quad\quad {\bf Step 3.} Compute $x_{k+1} = P_X ( x_{k} + \alpha_k \, d(x_k, \tilde y_k, \xi_k) )$.
		\item[] {\bf End do}
		\par\bigskip\noindent
    	\end{algorithmic}
    \end{algorithm}

As usual in the stochastic gradient literature~\cite{LBottou_FECurtis_JNocedal_2018}, due to the lack of reasonable stopping criteria for stochastic algorithms, we do not include a stopping condition in Algorithm~\ref{alg:MOBSG} (and in any of the algorithms developed in this paper). The convergence theory for the~BSG method developed in~\cite{TGiovannelli_GKent_LNVicente_2022} comprehensively covers several inexact settings, including the inexact solution of the~LL problem and the use of noisy estimates of the gradients and Hessians involved. The convergence rates of the~BSG method have been derived in~\cite{TGiovannelli_GKent_LNVicente_2022} under the assumptions of non-convexity, strong convexity, and convexity of the true objective function~$f$. The convergence theory of the algorithms introduced for the smooth case (i.e., optimistic and risk-neutral formulations in Sections~\ref{sec:LL_multiobj_prob}-\ref{sec:adj_grad_LL_multiobj_prob}, respectively) can be obtained as an extension of the convergence results presented in~\cite{TGiovannelli_GKent_LNVicente_2022}. The development of the convergence theory for the non-smooth case (i.e., risk-averse formulation in Section~\ref{sec:adj_grad_LL_multiobj_prob_general}) is left for future work.

%%%%%%%%%%%%%%%%%%%%%%%%%%%%%%%%%%%%%%%%%%%%%%%%%%%%%%%%%%%%%%%%%%%%%%%%%%%%%%%%%%%%%%%%%
\section{Basic definitions, results, and general assumptions}\label{sec:definitions}
%%%%%%%%%%%%%%%%%%%%%%%%%%%%%%%%%%%%%%%%%%%%%%%%%%%%%%%%%%%%%%%%%%%%%%%%%%%%%%%%%%%%%%%%%
{Given~$x \in X$, let us now focus on the multi-objective optimization problem given by the~LL problem in~\eqref{prob:single_bilevel}.
% which we rewrite here for convenience:
% \begin{equation}\label{prob:multiobj}
%     \begin{aligned}
%         \min_{y \in Y(x)} \ & F_{\ell}(x,y)=\left(f_{\ell}^{1}(x,y), \ldots, f_{\ell}^{q}(x,y)\right). 
%         % \text { s.t.} & \quad x \in \mathcal{X},
%     \end{aligned}
% \end{equation}
When the~LL objective functions are conflicting, minimizing one objective results in worse values for the others. Therefore, there is typically no single optimal solution that minimizes all objective functions simultaneously. In such cases, one is interested in obtaining a set of points where the value of one objective cannot be improved without deteriorating the values of the other objectives. Points with this property are called Pareto minimizers (or efficient or non-dominated points) and are introduced in Definitions~\ref{def:dominance}-\ref{def:pareto_point} below, which adapt the standard definitions in~\cite{MEhrgott_2005} to the~LL problem of~\eqref{prob:single_bilevel}.
\begin{definition}[LL Pareto dominance]\label{def:dominance}
    Given any two points~$\{y_1,y_2\} \subset \mathbb{R}^m$, we say that~$y_1$ dominates~$y_2$ if~$F_{\ell}(x,y_1) < F_{\ell}(x,y_2)$ componentwise. Moreover, we say that~$y_1$ weakly dominates~$y_2$ if~$F_{\ell}(x,y_1) \le F_{\ell}(x,y_2)$ componentwise and~$F_{\ell}(x,y_1) \ne F_{\ell}(x,y_2)$.
\end{definition}
\begin{definition}[LL Pareto minimizer]\label{def:pareto_point}
   A point~$y_* \in \mathbb{R}^m$ is a strict Pareto minimizer for the~LL problem of~\eqref{prob:single_bilevel} if no other point~$\bar{y} \in \mathbb{R}^m$ exists such that $y_*$ is weakly dominated by~$\bar{y}$. A point~$y_* \in \mathbb{R}^m$ is a weak Pareto minimizer if no other point~$\bar{y} \in \mathbb{R}^m$ exists such that $y_*$ is dominated by~$\bar{y}$.
\end{definition}

Let~$P_s(x)$ denote the set of strict~LL Pareto minimizers and~$P(x)$ the set of weak~LL Pareto minimizers. We wish to highlight that $P(x)$ represents $\Psi(x)$ from~\eqref{eq:single_multi_bilevel}.
Mapping the set~$P(x)$ into the objective space~$\mathbb{R}^m$ leads to the~LL Pareto front, which is defined as~$\{F_{\ell}(x,y)~:~y \in P(x)\}$. Note that Definition~\ref{def:pareto_point} implies that~$P_s(x) \subseteq P(x)$. 
% The well-definedness of the~LL problem of~\eqref{prob:single_bilevel} can be guaranteed in a classical way. In particular, under the continuity assumption of $f_{\ell}^j$, for all~$j \in \{1, \ldots, q\}$, we assume that the set~$P(x)$ is non-empty and compact, and the set~$P_s(x)$ is non-empty.

%We now formally state the assumptions that we make on the differentiability of the UL and LL functions as we will require it to hold throughout the rest of this paper.
%\begin{assumption}[Continuous differentiability]\label{ass:cont_diff}
%    The UL function $f_u$ is once continuously differentiable and all the LL functions $f_\ell^j$, with $j\in\{1,...,q\}$, are twice continuously differentiable.
%\end{assumption}

To compute a Pareto front, one can use scalarization techniques to reduce a multi-objective problem into a single-objective one, which can then be solved using classical optimization approaches~\cite{MEhrgott_2005,KMiettinen_2012}. One popular scalarization technique is the weighted-sum method, which consists of weighting the objective functions into a single objective~$\sum_{j=1}^{q} \lambda_j f_{\ell}^j(x, y)$, where~$\lambda_j$ are non-negative weights, for all~$j \in \{1, \ldots, q\}$. 
A necessary and sufficient condition for weak~LL Pareto optimality based on the weighted-sum method is included in Proposition~\ref{prop:weighted_sum_method} below, along with a sufficient condition for equivalence between~$P_s(x)$ and~$P(x)$. We refer to~\cite{MEhrgott_2005,AMGeoffrion_1968,KMiettinen_2012} for the proof of such a proposition.

\begin{proposition}\label{prop:weighted_sum_method}
   Let the~LL objective functions~$f_{\ell}^1(x,\cdot), \ldots, f_{\ell}^q(x,\cdot)$ be convex for all $x\in X$. Then, $y_* \in P(x)$ if and only if there exist weights~$\lambda_j\geq0$, for all $j\in\{1,...,q\}$, not all zero, such that $y_* \in \argmin_{y \in \mathbb{R}^m} \sum_{j=1}^{q} \lambda_j f_{\ell}^j(x, y)$. Moreover, if the~LL objective functions~$f_{\ell}^1(x,\cdot), \ldots, f_{\ell}^q(x,\cdot)$ are strictly convex for a certain $x\in X$, then $P_s(x) = P(x)$.
\end{proposition}

For the remainder of the paper, we require Assumption~\ref{ass:convexity} below.
\begin{assumption}\label{ass:convexity}
    The~LL objective functions~$f_{\ell}^1(x,\cdot), \ldots, f_{\ell}^q(x,\cdot)$ are strictly convex for all $x\in X$. Further, the set~$P(x)$ is non-empty for all $x\in X$.
\end{assumption}

Under Assumption~\ref{ass:convexity}, it is clear from Proposition~\ref{prop:weighted_sum_method} that $P(x) = P_s(x)$. The following remark characterizes the non-emptiness of $P(x)$.

\begin{remark} \label{rem:PePs}
In the case $Y(x)=\mathbb{R}^m$ considered in this paper,
we point out that $P(x) \neq \emptyset$ if the LL objective functions~$f_{\ell}^1(x,\cdot), \ldots, f_{\ell}^q(x,\cdot)$ are uniformly convex, which further implies that $P(x)$ is also compact. When $Y(x)\neq\mathbb{R}^m$, we have that $P_s(x) \neq \emptyset$ (and thus $P(x) \neq \emptyset$) if the set $Y(x)$ is compact, in which case $P(x)$ is also compact. 
\end{remark}

%%%%%%%%%%%%%%%%%%%%%%%%%%%%%%%%%%%%%%%%%%%%%%%%%%%%%%%%%%%%%%%%%%%%%%%%%%%%%%%%%%%%%%%%%
\section{The known optimistic formulation}\label{sec:LL_multiobj_prob}
%%%%%%%%%%%%%%%%%%%%%%%%%%%%%%%%%%%%%%%%%%%%%%%%%%%%%%%%%%%%%%%%%%%%%%%%%%%%%%%%%%%%%%%%%

The so-called optimistic formulation of problem~\eqref{prob:single_bilevel} corresponds to the problem
\begin{equation}\label{prob:single_bilevel_pareto}
	\begin{split}
	\min_{x \in \mathbb{R}^n, \, y \in \mathbb{R}^m} ~~ & f_u(x,y) \\
	\mbox{s.t.}~~ & x \in X \\
 & y \in P(x),\\
	\end{split}
\end{equation}
where the set of~LL optimal solutions is given by the set of weak~LL Pareto minimizers~$P(x)$. In accordance with Proposition~\ref{prop:equivalence_opt} below, one can prove that problem~\eqref{prob:single_bilevel_pareto} is equivalent to
\begin{equation}\label{prob:weighted_MBM}
	\setlength{\jot}{5pt}
	\begin{split}
	\min_{x \in \mathbb{R}^n, \, \lambda \in \mathbb{R}^q, \, y \in \mathbb{R}^m} & ~~ f_u(x,y) \\
	\mbox{s.t.} & ~~ x \in X, \; \lambda \in \Lambda \\ 
	& ~~ y \in \argmin_{y \in \mathbb{R}^m} ~~ f_{\ell}(x,\lambda,y) := \lambda^\top F_\ell(x,y),
	\end{split}
\end{equation}
where~$\Lambda$ denotes the simplex set, i.e.,
\begin{equation}\label{eq:simplex_set}
\Lambda = \left\{\lambda \in \mathbb{R}^q ~\middle|~ \sum_{j=1}^{q} \lambda_j = 1, \, \lambda_j \ge 0 \ \forall j \in \{1, \ldots, q\} \right\}.
\end{equation}
Given~$(x,\lambda) \in X \times \Lambda$, let~$\Phi(x,\lambda) = \{y \in \mathbb{R}^m ~|~ y \in \argmin_{y \in \mathbb{R}^m} f_{\ell}(x,\lambda,y) \}$ denote the set of optimal solutions to the~LL problem in~\eqref{prob:weighted_MBM}. The equivalence between problems~\eqref{prob:single_bilevel_pareto} and~\eqref{prob:weighted_MBM} is stated below (assuming implicitly that each problem admits an optimal solution).

\begin{proposition}\label{prop:equivalence_opt}
Let Assumption~\ref{ass:convexity} hold. If~$(\bar{x},\bar{y})$ is an optimal solution to problem~\eqref{prob:single_bilevel_pareto}, then, for all~$\bar{\lambda} \in \Lambda$ such that~$\bar{y} \in \Phi(\bar{x},\bar{\lambda})$, the point~$(\bar{x},\bar{\lambda},\bar{y})$ is an optimal solution to problem~\eqref{prob:weighted_MBM}. Moreover, 
if~$(\bar{x},\bar{\lambda},\bar{y})$ is an optimal solution to problem~\eqref{prob:weighted_MBM}, then~$(\bar{x},\bar{y})$ is an optimal solution to problem~\eqref{prob:single_bilevel_pareto}. 
\end{proposition}

\begin{proof} 
Let~$(\bar{x},\bar{y})$ be an optimal solution to problem~\eqref{prob:single_bilevel_pareto}. We have~$(\bar{x}, \bar{y}) \in X \times P(\bar{x})$ and~$f_u(\bar{x},\bar{y}) \le f_u(x,y)$ for all~$(x,y) \in X \times P(x)$. Assume that there exists~$\lambda^0 \in \Lambda$ with~$\bar{y} \in \Phi(\bar{x},\lambda^0)$ such that~$(\bar{x},\lambda^0,\bar{y})$ is not an optimal solution to problem~\eqref{prob:weighted_MBM}. Therefore, there exists~$(\hat{x},\hat{\lambda}) \in X \times \Lambda$, with~$\hat{y} \in \Phi(\hat{x},\hat{\lambda})$, such that~$f_u(\bar{x},\bar{y}) > f_u(\hat{x},\hat{y})$. Note that from Proposition~\ref{prop:weighted_sum_method}, since~$\hat{y} \in \Phi(\hat{x},\hat{\lambda})$, it follows that~$\hat{y} \in P(\hat{x})$, which implies that $(\hat{x},\hat{y})$ is a feasible point for problem~\eqref{prob:single_bilevel_pareto}. All these facts contradict the optimality of~$(\bar{x},\bar{y})$ for problem~\eqref{prob:single_bilevel_pareto}.

    Vice versa, let~$(\bar{x},\bar{\lambda},\bar{y})$ be an optimal solution to problem~\eqref{prob:weighted_MBM}. We have~$(\bar{x},\bar{\lambda}) \in X \times \Lambda$, with $\bar{y} \in \Phi(\bar{x},\bar{\lambda})$, and~$f_u(\bar{x},\bar{y}) \le f_u(x,y)$ for all~$(x,\lambda) \in X \times \Lambda$, with~$y \in \Phi(x,\lambda)$. From Proposition~\ref{prop:weighted_sum_method}, since~$\bar{y} \in \Phi(\bar{x},\bar{\lambda})$, it follows that~$\bar{y} \in P(\bar{x})$, which implies that $(\bar{x},\bar{y})$ is a feasible point for problem~\eqref{prob:single_bilevel_pareto}. Assume that~$(\bar{x},\bar{y})$ is not an optimal solution to problem~\eqref{prob:single_bilevel_pareto}. Therefore, there exists~$(\hat{x},\hat{y}) \in X \times P(\hat{x})$ such that~$f_u(\bar{x},\bar{y}) > f_u(\hat{x},\hat{y})$. From Proposition~\ref{prop:weighted_sum_method}, since~$\hat{y} \in P(\hat{x})$, there exists~$\hat{\lambda} \in \Lambda$ such that $\hat{y} \in \Phi(\hat{x},\hat{\lambda})$, which implies that~$(\hat{x},\hat{\lambda},\hat{y})$ is a feasible point for problem~\eqref{prob:weighted_MBM}. All these facts contradict the optimality of~$(\bar{x},\bar{\lambda},\bar{y})$ for problem~\eqref{prob:weighted_MBM}.
\end{proof}

Note that the~LL objective function in problem~\eqref{prob:weighted_MBM} is~$f_{\ell}(x, \lambda, y) = \sum_{j = 1}^{q} \lambda_j f_{\ell}^j(x,y)$, where~$\lambda \in \Lambda$. Hence, problem~\eqref{prob:weighted_MBM} is a bilevel problem where both levels are single-objective functions and can be solved by applying the BSG method introduced in Section~\ref{sec:multilevel_multiobj_alg} (when the functions are stochastic).
Let us denote the optimal solution of the~LL problem in~\eqref{prob:weighted_MBM} by~$y(x, \lambda)$. For the calculation of the~BSG direction, we require Assumption~\ref{ass:ass_bsg} below, which ensures the existence of~$y(x, \lambda)$ for all~$\lambda \in \Lambda$.
To do this, we formally assume a certain smoothness of the functions we are dealing with.

\begin{assumption}\label{ass:cont_diff}
    The UL function $f_u$ is once continuously differentiable and all the LL functions $f_\ell^j$, with $j\in\{1,...,q\}$, are twice continuously differentiable.
\end{assumption}

We remark that Assumptions~\ref{ass:convexity} and \ref{ass:cont_diff} together imply that the Hessians~$\nabla_{yy}^2 f_{\ell}^j(x, y)$, for all~$j \in \{1, \ldots, q\}$, are positive definite for all~$x \in X$.

Further, we also assume the existence of a solution to the LL problem. We state here this requirement for both the optimistic and risk-neutral formulations together (although we formally introduce the risk-neutral case in Section~\ref{sec:adj_grad_LL_multiobj_prob}) to avoid repetition. It bears mentioning that this assumption does not encompass the risk-averse case here, as it requires a different approach, which we introduce in Section~\ref{sec:adj_grad_LL_multiobj_prob_general}.

\begin{assumption}[Existence of LL solution]\label{ass:ass_bsg}
    For any~$x \in X$ and~$\lambda \in \Lambda$, there exists a point~$y(x, \lambda)$ such that $\nabla_y f_{\ell} (x, \lambda, y(x, \lambda)) = \sum_{j=1}^{q} \lambda_j \nabla_y f_{\ell}^j(x, y(x,\lambda)) = 0$. Moreover, the stochastic estimates of the Hessians~$\nabla_{yy}^2 f_{\ell}^j(x, y)$, for all~$j \in \{1, \ldots, q\}$, are positive definite at all points.
\end{assumption}

Given the strict convexity of the LL functions imposed in Assumption~\ref{ass:convexity}, it then becomes clear under Assumption~\ref{ass:ass_bsg} that the 
point~$y(x, \lambda)$ is the unique solution of the LL problem.

Let $f_{\OPT}(x, \lambda) = f_u(x, y(x,\lambda))$, with $\OPT$ standing for optimistic. By applying the chain rule to~$f_u (x,y(x,\lambda))$, one obtains the gradient vectors
\begin{equation}  \label{eq:02}
    \nabla_x f_{\OPT} \; = \; \nabla_x f_u + \nabla_x y \nabla_y f_u, \qquad
    \nabla_{\lambda} f_{\OPT} \; = \; \nabla_{\lambda} y \nabla_y f_u. 
\end{equation}

To calculate the Jacobian $\nabla y(x,\lambda)\in\mathbb{R}^{(n+q)\times m}$, we take derivatives with respect to~$x$ and~$\lambda$ on both sides of the~LL first-order necessary optimality conditions $\nabla_y f_\ell(x,\lambda,y(x,\lambda))=\sum_{j=1}^{q} \lambda_j \nabla_y f^j_\ell(x,y(x,\lambda))=0$, yielding the equations
\begin{equation}  \label{eq:04}
    \nabla^2_{y x} f_{\ell} + \nabla^2_{y y} f_{\ell} \nabla_x y^{\top} \; = \; 0, \qquad
     \nabla^2_{y \lambda} f_{\ell} + \nabla^2_{y y} f_{\ell} \nabla_{\lambda} y^{\top} \; = \; 0.
\end{equation}
Again, the differentiability of~$y(\cdot)$ with respect to~$x$ and~$\lambda$ is a consequence of the implicit function theorem~\cite{Ruding_1953}. Under Assumption~\ref{ass:ass_bsg}, we can obtain~$\nabla_{x} y$ and~$\nabla_{\lambda} y$ from~\eqref{eq:04} and plug their values into~\eqref{eq:02}, which leads to
\begin{equation}  \label{eq:044}
     \nabla_x f_{\OPT} \; = \; \nabla_x f_u - \nabla_{xy}^2 f_\ell (\nabla_{yy}^2 f_\ell)^{-1} \nabla_y f_u, \qquad \nabla_{\lambda} f_{\OPT} \; = \; - \nabla_{\lambda y}^2 f_\ell (\nabla_{yy}^2 f_\ell)^{-1} \nabla_y f_u,
\end{equation}
where all gradients and Hessians are evaluated at $(x,y(x,\lambda))$. Note that
\begin{alignat}{1}\label{eq:222}
        \nabla_{xy}^2 f_{\ell} = \sum_{j = 1}^{q} \lambda_j \nabla_{xy}^2 f_{\ell}^j, \quad 
        \nabla_{yy}^2 f_{\ell} = \sum_{j = 1}^{q} \lambda_j \nabla_{yy}^2 f_{\ell}^j, \quad
        \nabla_{\lambda y}^2 f_{\ell} = \left(\nabla_y f_{\ell}^1, \ldots, \nabla_y f_{\ell}^q\right)^{\top}.
\end{alignat}  
Therefore, from~\eqref{eq:044}--\eqref{eq:222}, one can obtain the adjoint gradient~$\nabla f_{\OPT}$ by concatenating the subvectors~$\nabla_x f_{\OPT}$ and~$\nabla_{\lambda} f_{\OPT}$ into a single vector. In the stochastic case, all the gradients and Hessians on the right-hand sides of~\eqref{eq:044} can be replaced by corresponding stochastic estimates.

In Algorithm~\ref{alg:BSMG_MOLL_opt}, we introduce a bilevel stochastic gradient method to solve the optimistic formulation of problem~\eqref{prob:single_bilevel}, given by problem~\eqref{prob:weighted_MBM}. Note that the main differences between Algorithm~\ref{alg:BSMG_MOLL_opt} and the classical~BSG  method reported in Algorithm~\ref{alg:MOBSG} are Step~2, where one now computes a (negative) BSG~$d\left(x_k,\lambda_k,\Tilde{y}_k,\xi_k\right)$ to approximate~$-\nabla f_{\OPT} = -(\nabla_x f_{\OPT},\nabla_{\lambda} f_{\OPT})$, and Step~3, where the orthogonal projection is now applied to the $\lambda$ variables as well.

\begin{algorithm}[H]
\caption{BSG-OPT Method}\label{alg:BSMG_MOLL_opt}
\begin{algorithmic}[1]
		\medskip
		\item[] {\bf Input:} $(x_0,\lambda_0,y_0) \in \mathbb{R}^n \times \mathbb{R}^q \times \mathbb{R}^m$, $\{\alpha_k\}_{k \geq 0} > 0$.
		\medskip
		\item[] {\bf For $k = 0, 1, 2, \ldots$ \bf do}
		\item[] \quad\quad {\bf Step 1.}
		 Obtain an approximation $\tilde{y}_k$ to the LL optimal solution $y(x_k,\lambda_k)$.
 \nonumber
    \item[] \quad\quad {\bf Step 2.}
		 Compute a negative BSG $d(x_k,\lambda_k,\tilde{y}_k,\xi_k)$.
		\item[] \quad\quad {\bf Step 3.} Compute $\left(x_{k+1},\lambda_{k+1}\right) = P_{X\Lambda} ( \left(x_{k},\lambda_k\right) + \alpha_k \, d(x_k, \lambda_k, \tilde y_k, \xi_k) )$, where $P_{X\Lambda}$ projects the $x$ and $\lambda$ variables onto the feasible regions $X$ and $\Lambda$, respectively.
		\item[] {\bf End do}
		\par\bigskip\noindent
\end{algorithmic}
\end{algorithm}

Note that, in principle, every point~$y\in P(x)$ can be considered an~LL optimal solution. Different LL~Pareto points have a different impact on the~UL objective function and, accordingly, they can lead to different optimal solutions to the bilevel problem. However, by using the optimistic formulation~\eqref{prob:weighted_MBM}, only the LL~Pareto point that is most favorable for the~UL objective function is selected among all the points~$y\in P(x)$. Thus, we will now consider new alternative approaches to the optimistic formulation~\eqref{prob:weighted_MBM}.

%%%%%%%%%%%%%%%%%%%%%%%%%%%%%%%%%%%%%%%%%%%%%%%%%%%%%%%%%%%%%%%%%%%%%%%%%%%%%%%%%%%%%%%%%
\section{A new risk-neutral formulation }\label{sec:adj_grad_LL_multiobj_prob}
%%%%%%%%%%%%%%%%%%%%%%%%%%%%%%%%%%%%%%%%%%%%%%%%%%%%%%%%%%%%%%%%%%%%%%%%%%%%%%%%%%%%%%%%%

In this section, we introduce a new formulation for problem~\eqref{prob:single_bilevel}. To gain intuition, suppose that the set of weak~LL Pareto minimizers is the same for all feasible values of $x$, i.e., $P(x) = P$ for all~$x \in X$. By interpreting~$y$ as a parameter, one can consider the parametric optimization problem 
\begin{equation}\label{prob:param_prob}
	\setlength{\jot}{5pt}
	\begin{rcases}
	\min_{x \in \mathbb{R}^n} & f_u(x,y) \\
	~~ \mbox{s.t.}  & x \in X\\
	\end{rcases}
	~~ \text{ with }y \in P,
\end{equation}
which can be addressed by considering two approaches in addition to the optimistic one. The {\it risk-neutral} approach assumes that~$y$ is a random vector with a probability distribution defined over~$P$ and considers the formulation
\begin{equation} \label{prob:param_average_approach}
\min_{x \in X} ~~ \mathbb{E}_{y \sim P}\left[ f_u(x,y) \right],
\end{equation}
where~$\mathbb{E}_{y \sim P}$ denotes the expected value that is taken with respect to the distribution of~$y$ over the domain~$P$. The objective function of problem~\eqref{prob:param_average_approach} can be approximated by using a sample mean, and the resulting problem can be solved by applying the~SG method.

In the risk-neutral formulation for the general case with $P(x)$, given~$x$ and assuming that~$y$ is a random vector with a probability distribution defined over~$P(x)$, the problem to solve is
\begin{equation} \label{prob:average_approach_exp_0}
\min_{x \in X} ~~ \mathbb{E}_{y \sim P(x)}\left[ f_u(x,y) \right].
\end{equation}

Under Assumption~\ref{ass:convexity}, we consider a companion problem to~\eqref{prob:average_approach_exp_0} that will provide us with a tractable solution procedure. In particular, assuming that~$\lambda$ is a continuous random vector with a probability distribution defined over~$\Lambda$ (defined in~\eqref{eq:simplex_set}), the problem we will consider instead is
\begin{equation} \label{prob:average_approach_exp}
\min_{x \in X} ~~ \mathbb{E}_{\lambda \sim \Lambda}\left[ f_u(x,y(x,\lambda)) \right],
\end{equation}
where~$y(x,\lambda)$ now denotes the optimal solution of the problem
\begin{equation} \label{prob:average_approach_exp_2}
\min_{y \in \mathbb{R}^m} ~~ F_{\ell}(x, y)^{\top} \lambda.
\end{equation}
Again, Assumptions~\ref{ass:convexity}, \ref{ass:cont_diff}, and~\ref{ass:ass_bsg} all together imply that~(\ref{prob:average_approach_exp_2}) has a unique solution.

Under Assumptions~\ref{ass:convexity}, \ref{ass:cont_diff}, and~\ref{ass:ass_bsg}, one can also calculate the~BSG direction for the risk-neutral formulation. 
In practice, one can consider a finite set~$\Lambda_{\RN} = \{\lambda^{1}, \ldots, \lambda^{N}\} \subset \Lambda$ that corresponds to a fine-scale discretization of~$\Lambda$, with~$\RN$ standing for risk-neutral. The corresponding set of weak~LL Pareto minimizers is~$\{y(x,\lambda^1),\ldots,y(x,\lambda^N)\} \subset P(x)$. Therefore, problem~\eqref{prob:average_approach_exp} can be approximated as
\begin{equation} \label{prob:average_approach_approx}
\min_{x \in X} ~~ f_{\RN}(x) \; = \;\frac{1}{N}\sum_{i=1}^{N}f_u\left(x,y(x,\lambda^i)\right).
\end{equation}
Note that the gradient of the objective function in~\eqref{prob:average_approach_approx} is given by
\begin{equation}\label{eq:grad_f}
    \nabla f_{\RN}(x) \; = \; \frac{1}{N}\sum_{i=1}^{N} \nabla f^i_{\RN}(x),
\end{equation}
where~$f^i_{\RN}(x) = f_u(x,y(x,\lambda^i))$ for all~$i \in \{1,\ldots,N\}$.

By applying the chain rule to~$f_u(x,y(x,\lambda^i))$, one obtains
\begin{equation} \label{adjoint_jacobian}
\nabla f^i_{\RN} \; = \; \nabla_x f_u + \nabla_{x} y(x,\lambda^i) \nabla_y f_u.
\end{equation} 
Then, the Jacobian~$\nabla_{x} y(x,\lambda^i)\in\mathbb{R}^{n\times m}$ can be calculated through the first-order~LL Pareto necessary optimality condition~$\sum_{j=1}^{q} \lambda_j^i \nabla_y f^j_\ell(x,y(x,\lambda^i))=0$, where~$\lambda_j^i$ denotes the $j$-th component of the vector~$\lambda^i$, for all~$i \in \{1, \ldots, N\}$. In particular, by taking the derivative of both sides of this equation with respect to $x$, using the chain rule and the implicit function theorem, we obtain
\begin{equation}\label{eq:221}
\sum_{j=1}^{q} \lambda_j^i \left(\nabla^2_{yx}f_\ell^j + \nabla^2_{yy}f^j_\ell\nabla_{x} y(x,\lambda^i)^{\top}\right) = 0,
\end{equation}
where all Hessians are evaluated at~$(x,y(x,\lambda^i))$. 
We recall here that Assumptions~\ref{ass:convexity} and \ref{ass:cont_diff} together imply that any convex combination of the Hessians~$\nabla_{yy}^2 f_{\ell}^j(x, y)$, $j \in \{1, \ldots, q\}$, is positive definite, and thus non-singular.
Equation~\eqref{eq:221} yields 
\begin{equation}\label{eq:jacobian_formula}
\nabla_{x} y(x,\lambda^i) \; = \; - \left(\sum_{j=1}^{q} \lambda_j^i\nabla^2_{xy}f^j_\ell \right) \left(\sum_{j=1}^{q} \lambda_j^i \nabla^2_{yy}f^j_\ell \right)^{-1}, \text{ for all $i \in \{1, \ldots, N\}$}.
\end{equation}
Plugging~\eqref{eq:jacobian_formula} into~\eqref{adjoint_jacobian}, one obtains
\begin{equation} \label{adjoint_2}
\nabla f^i_{\RN} \; = \; \nabla_x f_u - \left(\sum_{j=1}^{q} \lambda_j^i\nabla^2_{xy}f^j_\ell \right) \left(\sum_{j=1}^{q} \lambda_j^i \nabla^2_{yy}f^j_\ell \right)^{-1} \nabla_y f_u,
\end{equation}
where all gradients and Hessians on the right-hand side are evaluated at~$(x,y(x,\lambda^i))$. In the stochastic case, all the gradients and Hessians on the right-hand side of~\eqref{adjoint_2} can be replaced by corresponding stochastic estimates.

Since the number of elements in~$\Lambda_{\RN}$ can be significantly large, one can apply~SG to solve problem~\eqref{prob:average_approach_approx} by randomly choosing a set of samples (i.e., a mini-batch) from~$\Lambda_{\RN}$. Denoting a mini-batch as~$\Lambda_Q = \{\lambda^1, \ldots, \lambda^Q\} \; \subseteq\; \Lambda_{\RN}$, where~$Q$ is the mini-batch size, the corresponding sample set of~LL Pareto minimizers~$\{y(x,\lambda^1),\ldots,y(x,\lambda^Q)\}$ can be used to compute a~SG for~$\nabla f_{\RN}$ in~\eqref{eq:grad_f}. 

In Algorithm~\ref{alg:BSMG_MOLL_avg}, we introduce a bilevel stochastic gradient method to solve the risk-neutral interpretation of problem~\eqref{prob:single_bilevel} given by the formulation~\eqref{prob:average_approach_approx}. We adopt~$(\xi_{i})_k$ to denote the random variables used to obtain stochastic estimates for the~UL and~LL gradients and Hessians, for all~$i \in \{1, \ldots, Q\}$. For the sake of simplicity, at Steps~2--3, we denote~$\tilde{y}_k = \{(\tilde{y}(x_k,\tilde{\lambda}^i)~|~i \in \{1, \ldots, Q\}\}$ and~$\xi_k = \{(\xi_{i})_k~|~i \in \{1, \ldots, Q\}\}$. 
    
\begin{algorithm}[H]
\caption{BSG-RN Method}\label{alg:BSMG_MOLL_avg}
\begin{algorithmic}[1]
\medskip
		\item[] {\bf Input:} $x_0 \in \mathbb{R}^n$, $\{\alpha_k\}_{k \geq 0} > 0$.
		\medskip
		\item[] {\bf For $k = 0, 1, 2, \ldots$ \bf do}
\smallskip 
\item[] \quad\quad {\bf Step 1.} Obtain a mini-batch~$\{\tilde{\lambda}^1, \ldots, \tilde{\lambda}^Q\}$ and determine $\{\tilde{y}(x_k,\tilde{\lambda}^1),\ldots,\tilde{y}(x_k,\tilde{\lambda}^Q)\}$ to approximate the~LL Pareto optimal points~$\{y(x_k,\tilde{\lambda}^1),\ldots,y(x_k,\tilde{\lambda}^Q)\}$.
\smallskip
\item[] \quad\quad {\bf Step 2.}
	 Obtain a negative BSG $d(x_k,\tilde{y}(x_k,\tilde{\lambda}^i),(\xi_{i})_k)$ to approximate~$-\nabla f^i_{\RN}(x_k)$ for all~$i \in \{1,\ldots,Q\}$ and compute $d(x_k,\tilde{y}_{k},\xi_k) = (1/Q)\sum_{i=1}^{Q} d(x_k,\tilde{y}(x_k,\tilde{\lambda}^i),(\xi_{i})_k)$. 
\smallskip 
\item[] \quad\quad {\bf Step 3.} Compute $x_{k+1} = P_X ( x_{k} + \alpha_k \, d(x_k,\tilde{y}_{k},\xi_k) )$.
\nonumber
	\item[] {\bf End do}
		\par\bigskip\noindent
	\end{algorithmic}
\end{algorithm}

%%%%%%%%%%%%%%%%%%%%%%%%%%%%%%%%%%%%%%%%%%%%%%%%%%%%%%%%%%%%%%%%%%%%%%%%%%%%%%%%%%%%%%%%%
\section{A new risk-averse formulation}\label{sec:adj_grad_LL_multiobj_prob_general}
%%%%%%%%%%%%%%%%%%%%%%%%%%%%%%%%%%%%%%%%%%%%%%%%%%%%%%%%%%%%%%%%%%%%%%%%%%%%%%%%%%%%%%%%%

In this section, we introduce another new formulation for problem~\eqref{prob:single_bilevel}. Again, to gain intuition, suppose~$P(x) = P$ for all~$x \in X$, and consider problem~\eqref{prob:param_prob}.
The \textit{risk-averse} (or robust or pessimistic) formulation is given by the problem
\begin{equation} \label{prob:param_pessimistic_approach}
\min_{x \in X} \max_{y \in P}\; f_u(x,y), 
\end{equation}   
which can be reformulated as~$\min_{x \in X}\; f_{\RA}(x)$ by introducing~$f_{\RA}(x) = \max_{y \in P}\; f_u(x,y)$, with $\RA$ standing for risk-averse. 
Since~$f_{\RA}$ is nonsmooth, one can solve such a problem by applying a stochastic subgradient algorithm, where a subgradient can be obtained from the subdifferential of~$f_{\RA}$ at~$x$. Based on Danskin's Theorem~\cite{JMDanskin_1967}, such a subdifferential is given by~$\partial f_{\RA}(x) = \conv\{\nabla_x f_u(x,y) \; | \; y \in Y_0(x)\}$, where~$Y_0(x) = \{\bar{y} \in P \; | \; f_u(x,\bar{y}) = \max_{y \in P}\; f_u(x,y) \}$.

In the risk-averse formulation for the general case with $P(x)$, the problem to solve is
\begin{equation} \label{prob:param_pessimistic_approach_P(x)}
\min_{x \in X} \max_{y \in P(x)}\; f_u(x,y),
\end{equation}   
which can be reformulated as~$\min_{x \in X}\; f_{\RA}(x)$ by introducing~$f_{\RA}(x) = \max_{y \in P(x)}\; f_u(x,y)$. 
Under Assumptions~\ref{ass:convexity} and~\ref{ass:cont_diff}, let us specify~$P(x)$ algebraically by using the~LL first-order necessary and sufficient conditions for Pareto optimality (see Proposition~\ref{prop:weighted_sum_method}) as follows:
\begin{equation}\label{eq:def_P(x)}
P(x) \; = \; \left\{y \in \mathbb{R}^m~\middle|~\exists \lambda \in \Lambda \text{ such that } \sum_{j=1}^{q} \lambda_j \nabla_y f^j_\ell(x,y)=0\right\}.
\end{equation}

In accordance with Proposition~\ref{prop:equivalence_ra} below (where we will implicitly assume the well definedness of each problem), one can prove that Problem~\eqref{prob:param_pessimistic_approach_P(x)} is equivalent to the problem
\begin{equation} \label{prob:param_pessimistic_approach_P(x)_2}
\min_{x \in X} \max_{(y, \lambda) \in \tilde{P}(x)}\; f_u(x,y),
\end{equation}
where
\begin{equation}\label{eq:def_Ptilde(x)}
\tilde{P}(x) \; = \; \left\{(y, \lambda) \in \mathbb{R}^m \times \mathbb{R}^q ~\middle|~\lambda \in \Lambda \text{ and }\sum_{j=1}^{q} \lambda_j \nabla_y f^j_\ell(x,y)=0\right\}.
\end{equation}
Problem~\eqref{prob:param_pessimistic_approach_P(x)_2} can be reformulated as~$\min_{x \in X}\; \tilde{f}_{\RA}(x)$ by introducing the function $\tilde{f}_{\RA}(x) = \max_{(y, \lambda) \in \tilde{P}(x)}\; f_u(x,y)$. The 
optimization problem defining~$\tilde{f}_{\RA}(x)$ can be written as
\begin{equation}\label{prob:single_bilevel_pareto_practical}
	\begin{split}
	\max_{y \in \mathbb{R}^m, \, \lambda \in \mathbb{R}^q} ~~ & f_u(x,y) \\
	\mbox{s.t.}~~ & \lambda \in \Lambda, \\
 & \sum_{j=1}^{q} \lambda_j \nabla_y f^j_\ell(x,y)=0.
	\end{split}
\end{equation}
Note that from the definitions of~$P(x)$ and~$\tilde{P}(x)$ in~\eqref{eq:def_P(x)} and~\eqref{eq:def_Ptilde(x)}, respectively, we have
\begin{equation}\label{prop:equiv_P(x)_Ptilde(x)}
y \in P(x) \text{ if and only if } (y,\lambda) \in \tilde{P}(x) \text{ for some } \lambda \in \Lambda.
\end{equation}

\begin{proposition}\label{prop:equivalence_ra}
$\bar{x}$ is an optimal solution of problem~\eqref{prob:param_pessimistic_approach_P(x)} if and only if~$\bar{x}$ is an optimal solution of problem~\eqref{prob:param_pessimistic_approach_P(x)_2}.
\end{proposition}

\begin{proof}
Let~$\bar{x}$ be an optimal solution of problem~\eqref{prob:param_pessimistic_approach_P(x)}. Therefore, we can write~$f_{\RA}(\bar{x}) \le f_{\RA}(\tilde{x})$ for all~$\tilde{x} \in X$ or, equivalently,~$f_u(\bar{x},\bar{y}) \le f_u(\tilde{x},\tilde{y})$ for all~$\tilde{x} \in X$, where~$\bar{y} \in \argmax_{y \in P(\bar{x})} f_u(\bar{x},y)$ and~$\tilde{y} \in \argmax_{y \in P(\tilde{x})} f_u(\tilde{x},y)$. Now, assume that $\bar{x}$ is not a minimizer of the outer minimization problem in~\eqref{prob:param_pessimistic_approach_P(x)_2}. Then, there exists a point~$\hat{x} \in X$ such that~$\tilde{f}_{\RA}(\bar{x}) > \tilde{f}_{\RA}(\hat{x})$ or, equivalently,~$f_u(\bar{x},\bar{y}) > f_u(\hat{x},\hat{y})$, with~$\hat{y} \in \argmax_{(y,\lambda) \in \tilde{P}(\hat{x})} f_u(\hat{x},y)$. From~\eqref{prop:equiv_P(x)_Ptilde(x)} and from the fact that the objective function of the inner maximization problem in~\eqref{prob:param_pessimistic_approach_P(x)_2} does not depend on~$\lambda$, it follows that~$\hat{y} \in \argmax_{y \in P(\hat{x})} f_u(\hat{x},y)$. Therefore, $f_{\RA}(\bar{x}) > f_{\RA}(\hat{x})$, which contradicts the optimality of~$\bar{x}$ for problem~\eqref{prob:param_pessimistic_approach_P(x)}.

Now, let~$\bar{x}$ be an optimal solution of problem~\eqref{prob:param_pessimistic_approach_P(x)_2}. Therefore, $\tilde{f}_{\RA}(\bar{x}) \le \tilde{f}_{\RA}(\tilde{x})$ for all~$\tilde{x} \in X$ or, equivalently, $f_u(\bar{x},\bar{y}) \le f_u(\tilde{x},\tilde{y})$ for all~$\tilde{x} \in X$, where~$\bar{x} \in \argmax_{(y,\lambda) \in \tilde{P}(\bar{x})} f_u(\bar{x},y)$ and~$\tilde{y} \in \argmax_{(y,\lambda) \in \tilde{P}(\tilde{x})} f_u(\tilde{x},y)$. Again, from~\eqref{prop:equiv_P(x)_Ptilde(x)} and from the fact that the objective function of the inner maximization problem in~\eqref{prob:param_pessimistic_approach_P(x)_2} does not depend on~$\lambda$, we have~$\tilde{y} \in \argmax_{y \in P(\tilde{x})} f_u(\tilde{x},y)$. Therefore,~$f_{\RA}(\bar{x}) \le f_{\RA}(\tilde{x})$ for all~$\tilde{x} \in X$, which shows that~$\bar{x}$ is an optimal solution for problem~\eqref{prob:param_pessimistic_approach_P(x)}.
\end{proof}

For the calculation of a~BSG direction, and under Assumption~\ref{ass:cont_diff}, we start by introducing the Lagrangian function of problem~\eqref{prob:single_bilevel_pareto_practical} as~$\mathcal{L}(x,y,\lambda,z) = f_u(x,y) + z_{I}^{\top} \lambda + z_{E}^{\top} (\sum_{i=1}^{q} \lambda_i - 1, \sum_{j=1}^{q} \lambda_j \nabla_y f^j_\ell(x,y))$, where~$z_{I} \in \mathbb{R}^{q}$ and~$z_{E} \in \mathbb{R}^{m+1}$ are the vectors of Lagrange multipliers associated with the inequality and equality constraints, respectively. We will assume below the satisfaction of the first-order KKT conditions and the linear independence of the gradients of the active constraints (LICQ)~\cite{JNocedal_SJWright_2006} for problem~\eqref{prob:single_bilevel_pareto_practical}.

\begin{assumption}\label{ass:LICQ}
    For all~$x \in X$, there exists a $(y(x),\lambda(x))$ satisfying the first-order KKT conditions for problem~\eqref{prob:single_bilevel_pareto_practical} with associated multipliers $(z_I(x),z_E(x))$ such that the LICQ is satisfied.
\end{assumption}

Based on~\cite[Corollary 4.11]{AVFiacco_1983}, under Assumption~\ref{ass:LICQ} and additionally requiring~$\tilde{P}(x)$ to be non-empty and compact, the subdifferential of~$\tilde{f}_{\RA}$ at~$x$ is given by
\begin{equation}\label{eq:subdiff_Px}
\partial \tilde{f}_{\RA}(x) \; = \; \conv\{\nabla_x \mathcal{L}(x,y(x),\lambda(x),z(x,y(x),\lambda(x))) \; | \; (y(x), \lambda(x)) \in \tilde{Y}_0(x)\},
\end{equation}
where~$\tilde{Y}_0(x) = \{(\bar{y}(x), \bar{\lambda}(x)) \in \tilde{P}(x) \; | \; f_u(x,\bar{y}(x)) = \max_{(y, \lambda) \in \tilde{P}(x)}\; f_u(x,y) \}$ and~$z(x,y(x),\lambda(x))$ is the unique optimal vector of Lagrange multipliers (the uniqueness is a trivial consequence of Assumption~\ref{ass:LICQ}). One can ensure the compactness of~$\tilde{P}(x)$ from~\eqref{prop:equiv_P(x)_Ptilde(x)} by noticing that~$\Lambda$ is compact and requiring~$P(x)$ to be compact, which can be satisfied as suggested in Remark~\ref{rem:PePs}. The non-emptyness of~$\tilde{P}(x)$ (assumed in Assumption~\ref{ass:convexity} for $P(x)$) is ensured similarly; see also Remark~\ref{rem:PePs}.
Note that the gradient of the Lagrangian with respect to~$x$ is given by 
\begin{equation}\label{eq:1110}
    \nabla_x \mathcal{L}(x,y,\lambda,z) \; = \; \nabla_x f_u(x,y) + \sum_{j=1}^{q} \lambda_j \nabla_{xy}^2 f^j_\ell(x,y) z_{E}.
\end{equation} 
Equation~\eqref{eq:1110} can be used in~\eqref{eq:subdiff_Px}, where~$\tilde{Y}_0(x)$ is given by the solutions~$(\bar{y}(x),\bar{\lambda}(x))$ of problem~\eqref{prob:single_bilevel_pareto_practical}. In the stochastic case, all the gradients and Hessians on the right-hand side of~\eqref{eq:subdiff_Px}--\eqref{eq:1110} can be replaced by corresponding stochastic estimates. When using stochastic estimates, the subdifferential~\eqref{eq:subdiff_Px} is denoted as~$\partial \tilde{f}^d_{\RA}(x)$.

In Algorithm~\ref{alg:BSMG_MOLL_pess}, we introduce a bilevel stochastic subgradient method to solve the risk-averse formulation of problem~\eqref{prob:single_bilevel} given by problem~\eqref{prob:param_pessimistic_approach_P(x)}. We adopt~$\xi_k$ to denote the random variables used to obtain stochastic estimates for~UL and~LL gradients and Hessians.  

\begin{algorithm}[H]
\caption{BSG-RA Method}\label{alg:BSMG_MOLL_pess}
\begin{algorithmic}[1]
\medskip
\item[] {\bf Input:} $x_0 \in \mathbb{R}^n$, $\{\alpha_k\}_{k \geq 0} > 0$.
\medskip
\item[] {\bf For $k = 0, 1, 2, \ldots$ \bf do}
\item[] \quad\quad {\bf Step 1.} Obtain an approximation $(\tilde{y}_k, \tilde{\lambda}_k)$ to a solution~$(y(x_k), \lambda(x_k))$ of problem~\eqref{prob:single_bilevel_pareto_practical}.
\item[] \quad\quad {\bf Step 2.}
 Select a negative stochastic subgradient $d(x_k, \tilde{y}_k, \tilde{\lambda}_k, \xi_k) \in -\partial \tilde{f}^d_{\RA}(x_k)$. 
\item[] \quad\quad {\bf Step 3.} Compute $x_{k+1} = P_X ( x_{k} + \alpha_k \, d(x_k, \tilde{y}_k, \tilde{\lambda}_k, \xi_k) )$.
\nonumber
\item[] {\bf End do}
    \par\bigskip\noindent
\end{algorithmic}
\end{algorithm}

%%%%%%%%%%%%%%%%%%%%%%%%%%%%%%%%%%%%%%%%%%%%%%%%%%%%%%%%%%%%%%%%%%%%%%%%%%%%%%%%%%%%%%%%%
\section{Numerical experiments}\label{sec:numerical_exp}
%%%%%%%%%%%%%%%%%%%%%%%%%%%%%%%%%%%%%%%%%%%%%%%%%%%%%%%%%%%%%%%%%%%%%%%%%%%%%%%%%%%%%%%%%

%All tests were run using Dell Latitude 5520 (16GB of RAM, Intel(R) Core(TM) i7-1185G7 processor running at~3.00GHz).
All code was written in Python and the experimental results were obtained on a desktop computer (32GB of RAM, Intel(R) Core(TM) i9-9900K processor running at 3.60GHz).\footnote{All the code for our implementation is available at \url{https://github.com/GdKent/BMOLL_OPT_RN_RA}.}

%%%%%%%%%%%%%%%%%%%%%%%%%%%%%%%%%%%%%%%%%%%%%%%%%%%%%%%%%%%%%%%%%%%%%%%%%%%%%%%%%%%%%%%%%
\subsection{Our practical methods}\label{sec:practical_methods}
%%%%%%%%%%%%%%%%%%%%%%%%%%%%%%%%%%%%%%%%%%%%%%%%%%%%%%%%%%%%%%%%%%%%%%%%%%%%%%%%%%%%%%%%%

% A major difficulty in~\eqref{eq:044}, \eqref{adjoint_2}, and~\eqref{eq:1110} is the use of second-order derivatives of~$f^j_\ell$, which prevents the application of Algorithms~\ref{alg:BSMG_MOLL_opt}--\ref{alg:BSMG_MOLL_pess} to large-scale problems. In the practical algorithms that we propose, we avoid this problem by approximating the second-order derivatives with the outer products of the corresponding gradients, which have successfully been applied in~\cite{TGiovannelli_GKent_LNVicente_2022} for Algorithm~\ref{alg:MOBSG} and are inspired by Gauss-Newton methods for nonlinear least-squares problems. Such rank-1 approximations are as follows:
% \begin{alignat}{6}
%     & \nabla_{xy}^2 f_\ell^j \; &&\simeq \;  \nabla_x f_\ell^j \tcb{(\nabla_y f_\ell^{j})^{\top}}  \; &\text{ and } \; &\nabla_{yy}^2 f^j_\ell \; &&\simeq \; \nabla_y f^j_\ell \tcb{(\nabla_y f_\ell^{j})^{\top}}, \text{ for all }~j \in \{1, \ldots, q\}. \label{unconstrained_rank1_approx}
% \end{alignat}

% When using these rank-1 approximations, Algorithms~\ref{alg:BSMG_MOLL_opt}--\ref{alg:BSMG_MOLL_pess} will be referred to as~BSG-OPT-1, BSG-RN-1, and~BSG-RA-1, respectively, where the~``1'' stands for first-order rank-1 approximations of the Hessian matrices. 

In the numerical experiments, we tested Algorithms~\ref{alg:BSMG_MOLL_opt}--\ref{alg:BSMG_MOLL_pess} with both exact and stochastic Hessians in deterministic and stochastic settings, respectively. The resulting algorithms are referred to as~BSG-OPT-H, BSG-RN-H, and~BSG-RA-H, where the ``H'' stands for the Hessian matrix.
To deal with the inverse matrix in~\eqref{eq:044}, one could solve the linear system given by the adjoint equation $(\sum_{j=1}^q\lambda_j\nabla_{yy}^2 f_\ell^j) \mu = \nabla_y f_u$ for the variables~$\mu$, and then calculate~$\nabla_x f_{\OPT}$ from~$\nabla_x f_u - (\nabla^2_{xy} f_{\ell}) \mu$ and~$\nabla_{\lambda} f_{\OPT}$ from~$- (\nabla^2_{\lambda y} f_{\ell}) \mu$. A similar approach could be used to handle %~$(\sum_{j=1}^q\lambda_j^i\nabla_{yy}^2 f_\ell^j)^{-1}$ 
the inverse term in~\eqref{adjoint_2}.
Due to the small dimensions of the problems that we tested ($n,m\leq 50$), in~BSG-OPT and~BSG-RN, the adjoint systems are solved by factorizing the matrices $\sum_{j=1}^q\lambda_j\nabla_{yy}^2 f_\ell^j$ and $\sum_{j=1}^q\lambda_j^i\nabla_{yy}^2 f_\ell^j$, respectively. In practice, when the dimensions of the problems are large, one can solve the adjoint systems via the linear conjugate gradient method until non-positive curvature is detected.
%\textcolor{green}{In~BSG-OPT and~BSG-RN, the resulting adjoint systems are solved by the linear conjugate gradient method until non-positive curvature is detected}.
 In Step~1 of~BSG-OPT and~BSG-RN, we apply either gradient descent (in the deterministic setting) or stochastic gradient descent (in the stochastic setting) for a certain budget of iterations. The number of iterations increases by~1 every time the difference of the~UL objective function between two consecutive iterations is less than a given threshold. Such an increasing accuracy strategy has been used successfully in the~BSG method presented in~\cite{TGiovannelli_GKent_LNVicente_2022}. In Step~1 of BSG-RA, we solve problem~\eqref{prob:single_bilevel_pareto_practical} by applying the trust-region algorithm for nonlinear constrained problems proposed in~\cite{ARConn_NIMGould_PLToint_2000}. 

%%%%%%%%%%%%%%%%%%%%%%%%%%%%%%%%%%%%%%%%%%%%%%%%%%%%%%%%%%%%%%%%%%%%%%%%%%%%%%%%%%%%%%%%%
\subsection{Results for bilevel problems with a multi-objective lower level}\label{sec:synthetic_prob}
%%%%%%%%%%%%%%%%%%%%%%%%%%%%%%%%%%%%%%%%%%%%%%%%%%%%%%%%%%%%%%%%%%%%%%%%%%%%%%%%%%%%%%%%%

The set of problems that we tested are bilevel instances where the upper level is a quadratic single-objective problem and the lower level is a multi-objective problem. In particular, given~$h_1 \in \mathbb{R}^n$, $h_2 \in \mathbb{R}^m$, a symmetric positive definite matrix~$H_2 \in \mathbb{R}^{n \times n}$, and a matrix~$H_1 \in \mathbb{R}^{n \times m}$, we solve the general problem
    \begin{equation}\label{prob:synthetic_prob}
        \begin{split}
            \min_{x \in \mathbb{R}^n} ~  f_u(x, &y) \; = \; h_1^{\top} x + h_2^{\top} y + \frac{1}{2} x^{\top} H_1 y + \frac{1}{2} x^{\top} H_2 x, \\
            \text{s.t.} \ y \in &\argmin_{y \in \mathbb{R}^m} ~ F_{\ell}(x, y) = (f^1_{\ell}(x, y), f^2_{\ell}(x, y)), \\ 
        \end{split}
    \end{equation}
where the~LL objective functions considered in the experiments are specified in Table~\ref{tab:test_prob}, along with the reference for the~LL problem, the number of~UL and~LL variables (i.e.,~$n$ and~$m$, respectively), and the bounds on each~UL variable~$x_i$. The first two LL objective functions that we consider in our experiments, JOS1~\cite{YJin_MOlhofer_BSendhoff_2001} and~SP1~\cite{SHuband_etal_2006}, are both separable functions, i.e., they can be written as a sum of terms such that each variable only appears in one of the terms. As a result, the third LL objective, which we will refer to as~GKV1, leads to a more general multi-objective optimization problem that can be either separable or non-separable depending on the $H_1$, $H_2$, $H_3$, and $H_5$ matrices that are chosen.

\begin{table}
    \small
    \centering
    \begin{tabular}{ c|c|c|c|c|c } 
    Problem & $n$ & $m$ & Ref. for LL & LL Objective Functions & Bound on~$x_i$\\[2pt]
     \hline
    1 & $\bar{n}$ & $\bar{n}$ & JOS1~\cite{YJin_MOlhofer_BSendhoff_2001} & $f^1_{\ell}(x, y) = \frac{1}{\bar{n}}\sum_{i=1}^{\bar{n}} x_i^2 y_i^2$ & $[-2,\infty]$\\
     &  &  & &  $f^2_{\ell}(x, y) = \frac{1}{\bar{n}}\sum_{i=1}^{\bar{n}} (x_i - 2)^2(y_i - 2)^2$ &\\[2pt]
     \hline
    2 & $\bar{n}$ & $\bar{n}$ & SP1~\cite{SHuband_etal_2006} & $f^1_{\ell}(x, y) = \sum_{i=1}^{\bar{n}} [(x_i - 1)^2 + (x_i - y_i)^2]$ & $[-2,3]$\\
     &  &  & &  $f^2_{\ell}(x, y) = \sum_{i=1}^{\bar{n}} [(y_i - 3)^2 + (x_i - y_i)^2]$ &\\[2pt]
     \hline
    3 & $\bar{n}$ & $\bar{n}$ & GKV1 & $f^1_{\ell}(x, y) = \frac{1}{2} y^{\top} H_3 y - \frac{1}{2} y^{\top} H_4 x$ & $[-\infty,0]$ or $[0,\infty]$ \\
     &  &  & &  $f^2_{\ell}(x, y) = \frac{1}{2} y^{\top} H_5 y + \frac{1}{2} y^{\top} H_6 x$ &\\
    \end{tabular}
    \caption{Test problems ($\bar{n}$ is an arbitrary positive scalar).}\label{tab:test_prob}
\end{table}

In all the numerical experiments, we considered the same dimension at both the upper and lower levels (i.e.,~$n = m = \bar{n}$, with~$\bar{n}$ positive scalar) and we set~$H_1$ and~$H_2$ in~\eqref{prob:synthetic_prob} equal to identity matrices. The initial points were randomly generated according to a uniform distribution defined within the bounds specified in the last column of Table~\ref{tab:test_prob}. We compared all the algorithms by using either a line search~(LS) or a fixed stepsize~(FS) at both the~UL and~LL problems. We also considered a decaying stepsize but this led to worse performance and, therefore, we do not report the corresponding results. 
Recalling the set~$\Lambda_{\RN}$ introduced in Section~\ref{sec:adj_grad_LL_multiobj_prob}, when running~BSG-RN on problems with dimension~$\bar{n} > 1$, we use~$N = 500$ and~$Q = 20$ (see Figure~\ref{fig:GKV1_nonsep_BSGRN} for a comparison of the results obtained for different values of~$Q$). When~$\bar{n} = 1$, we use~$N = Q = 500$. For~BSG-OPT and~BSG-RN, we implemented an increasing accuracy strategy for the~LL problem by using~$f_u$ difference thresholds of~$0.1$ and~$0.9$, respectively, and a maximum number of LL iterations equal to~$30$. Note that one could also consider an increasing accuracy strategy for~BSG-RA to gradually improve the approximation of the solution of problem~\eqref{prob:single_bilevel_pareto_practical} obtained in Step~1 of Algorithm~\ref{alg:BSMG_MOLL_pess}. In this paper, for the practical algorithm considered for BSG-RA, we solve problem~\eqref{prob:single_bilevel_pareto_practical} by using the version of the trust-region method developed in~\cite{ARConn_NIMGould_PLToint_2000} available in the SciPy library~\cite{2020SciPy-NMeth}, with default parameters. In the figures, when comparing the algorithms in terms of iterations, we plot the true function values~$f_{\OPT}$ and $f_{\RN}$ for~BSG-OPT and~BSG-RN and an accurate approximation of the true function~$f_{\RN}$ for BSG-RA. 

We considered three sets of experiments corresponding to three different settings for the~LL problem: deterministic separable case, deterministic non-separable case, and stochastic non-separable case. In the latter case, the~UL problem is considered stochastic as well.

    \paragraph{Deterministic separable~LL case.} In this case, we consider LL~objective functions that are separable, i.e., all the problems from Table~\ref{tab:test_prob}. %, i.e., those functions can be written as the sum of terms such that each variable only appears in one term. In particular, we consider all the problems from Table~\ref{tab:test_prob}.
    In Problem~3, we set~$H_3$, $H_4$, $H_5$, and~$H_6$ equal to identity matrices and we consider the bounds on~$x_i$ given by~$[-\infty,0]$. In~\eqref{prob:synthetic_prob}, we set~$h_1$ and~$h_2$ equal to vectors of ones (except for Problem~3, where each element of~$h_1$ is equal to~$3$). For the problems in this case, we consider~$\Bar{n}=1$, which allows us to visualize the solution space in two dimensions for a more direct interpretation of the results.
    
    Figures~\ref{fig:JOS1}--\ref{fig:GKV1} show the results obtained by Algorithms~\ref{alg:BSMG_MOLL_opt}--\ref{alg:BSMG_MOLL_pess} when a backtracking Armijo line search~\cite{JNocedal_SJWright_2006} at both the~UL and~LL problems and exact Hessians are used. The~UL line search ensures a sufficient decrease of an accurate approximation of the true functions~$f_{\OPT}$, $f_{\RN}$, and $f_{\RA}$. We denote the~UL optimal solutions found by~BSG-OPT, BSG-RN, and BSG-RA as~$x_{\OPT}$, $x_{\RN}$, and $x_{\RA}$, respectively. Moreover, we denote the optimal solutions of the problems~$\min_{y \in P(x_{\OPT})} f_u(x_{\OPT},y)$ and~$\max_{y \in P(x_{\RA})} f_u(x_{\RA},y)$ as~$y_{\OPT}$ and~$y_{\RA}$, respectively. In each of these figures, in the upper left-hand plot, we compare the values of the true functions~$f_{\OPT}$, $f_{\RN}$, and $f_{\RA}$  achieved by each algorithm in terms of iterations. In the upper right-hand plot, we compare the sets~$\{(x^*,y)~|~y \in P(x)\}$, where~$x^*$ denotes the~UL optimal solution determined by each algorithm (i.e., $x^* \in \{x_{\OPT}, x_{\RN}, x_{\RA}\}$), and we also report the contour lines of the~UL objective function. In the lower plots, we compare the Pareto fronts between the~LL objective functions obtained for each~UL optimal solution in~$\{x_{\OPT}, x_{\RN}, x_{\RA}\}$, and we refer to the points~$(f_{\ell}^1, f_{\ell}^2)$ evaluated at~$(x_{\OPT}, y_{\OPT})$ and~$(x_{\RA}, y_{\RA})$ as the optimistic and pessimistic Pareto points, respectively. Note that the optimistic Pareto front {\it dominates} both the risk-neutral and risk-averse Pareto fronts in all the figures, although the three fronts correspond to different UL variable values.
    We point out that all the algorithms were able to find the optimal solutions to Problems~1--3.% regardless of the value assigned to~$\bar{n}$.

    \begin{figure}
    \centering
        \includegraphics[scale=0.38]{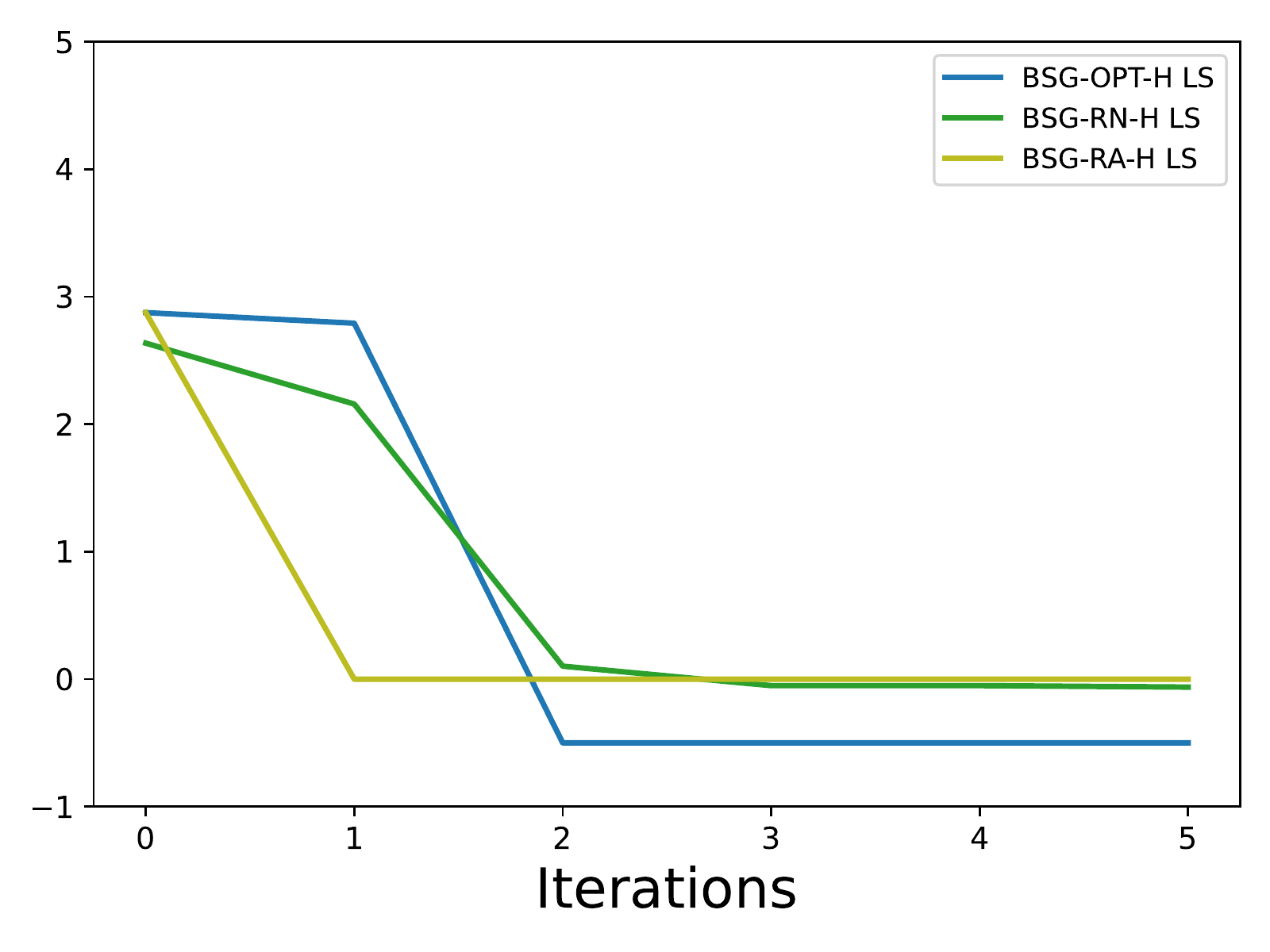}
          \includegraphics[scale=0.38]{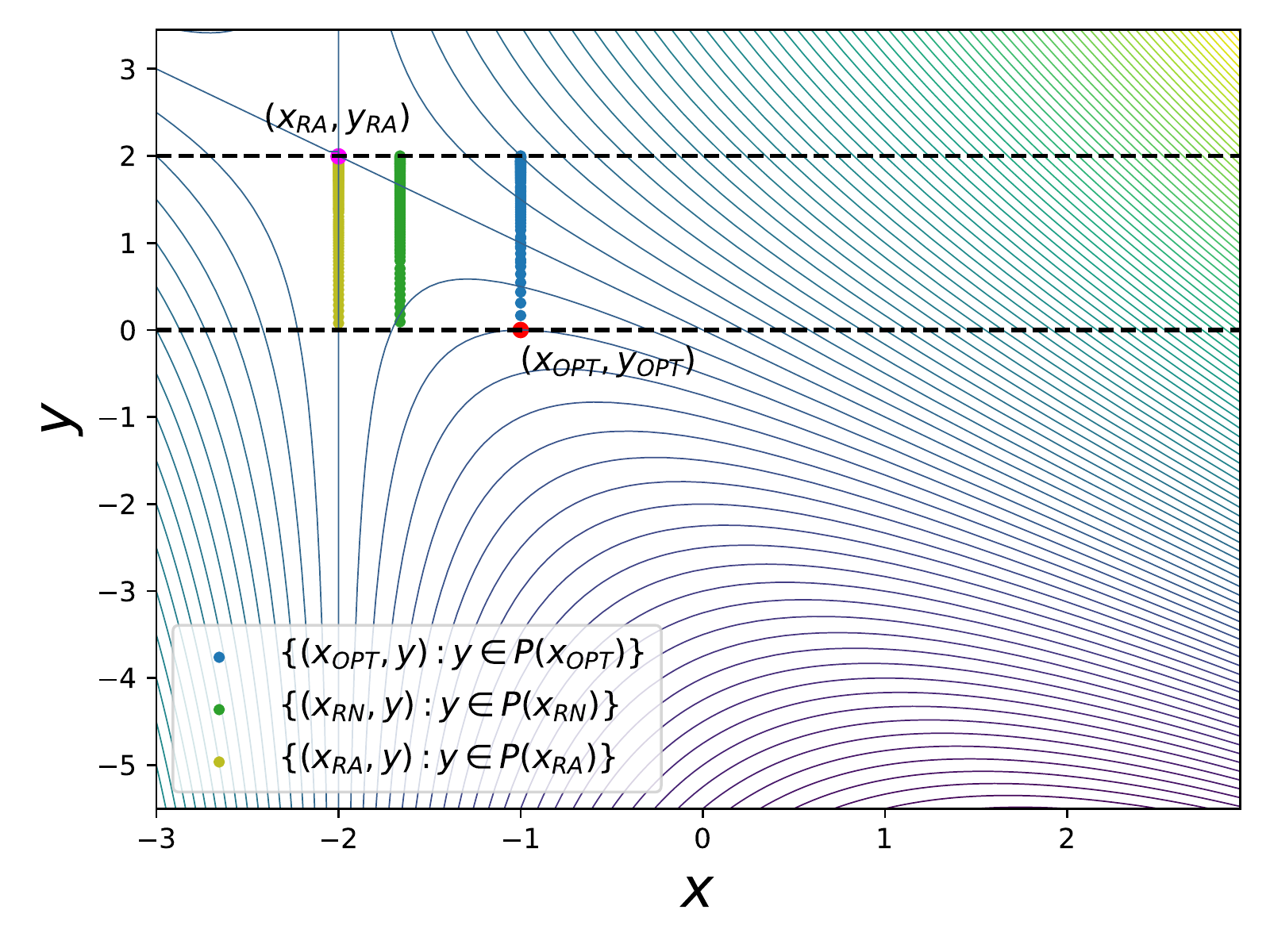} \includegraphics[scale=0.38]{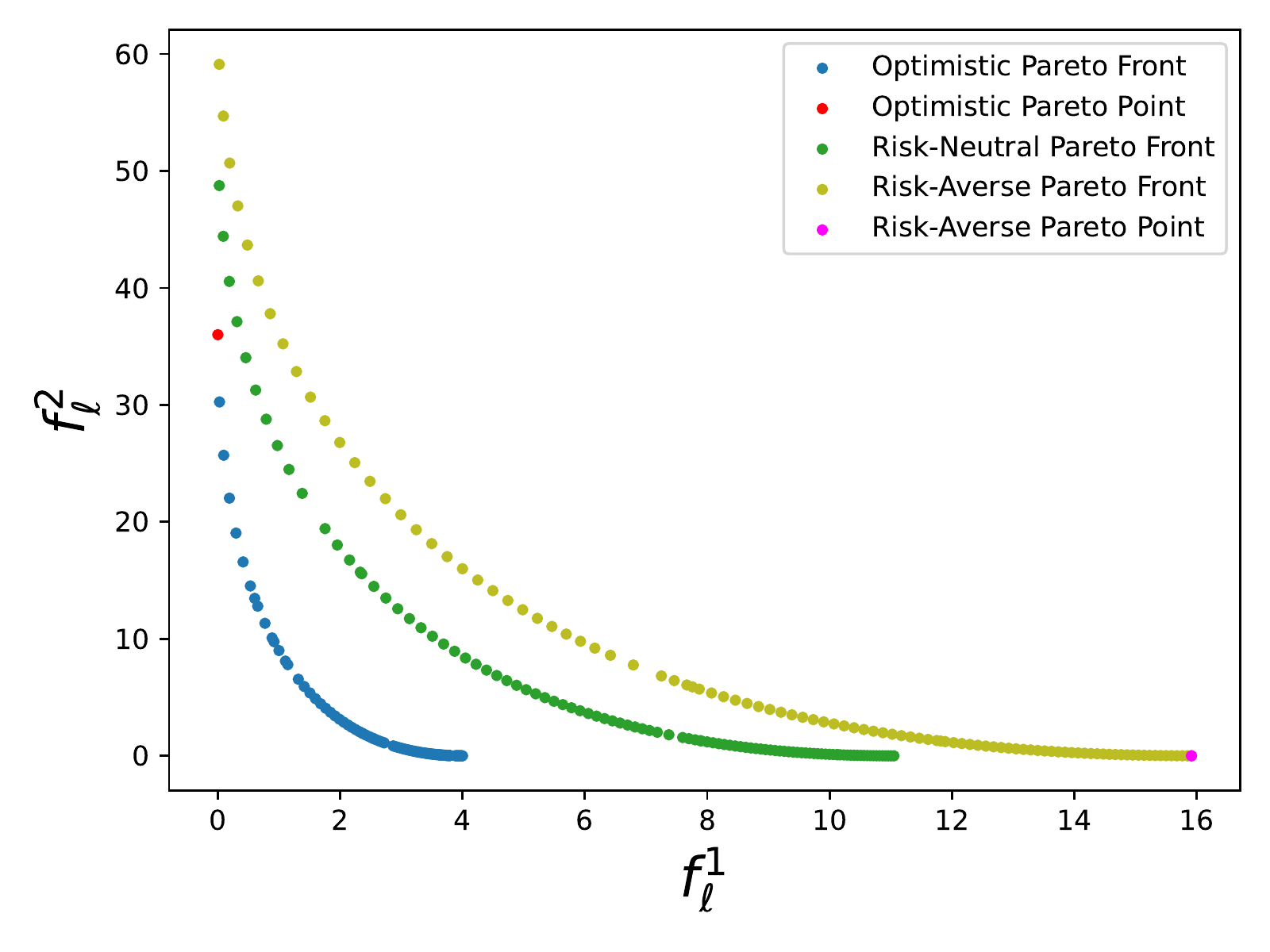} 
        \caption{Results for Problem~1 in the deterministic separable~LL case (for~$\bar{n}$ in Table~\ref{tab:test_prob} equal to~1). In the upper-left plot, the vertical axis represents the values of $f_{\OPT}$, $f_{\RN}$, and $f_{\RA}$.}\label{fig:JOS1}
    \end{figure}

    \begin{figure}
    \centering
        \includegraphics[scale=0.38]{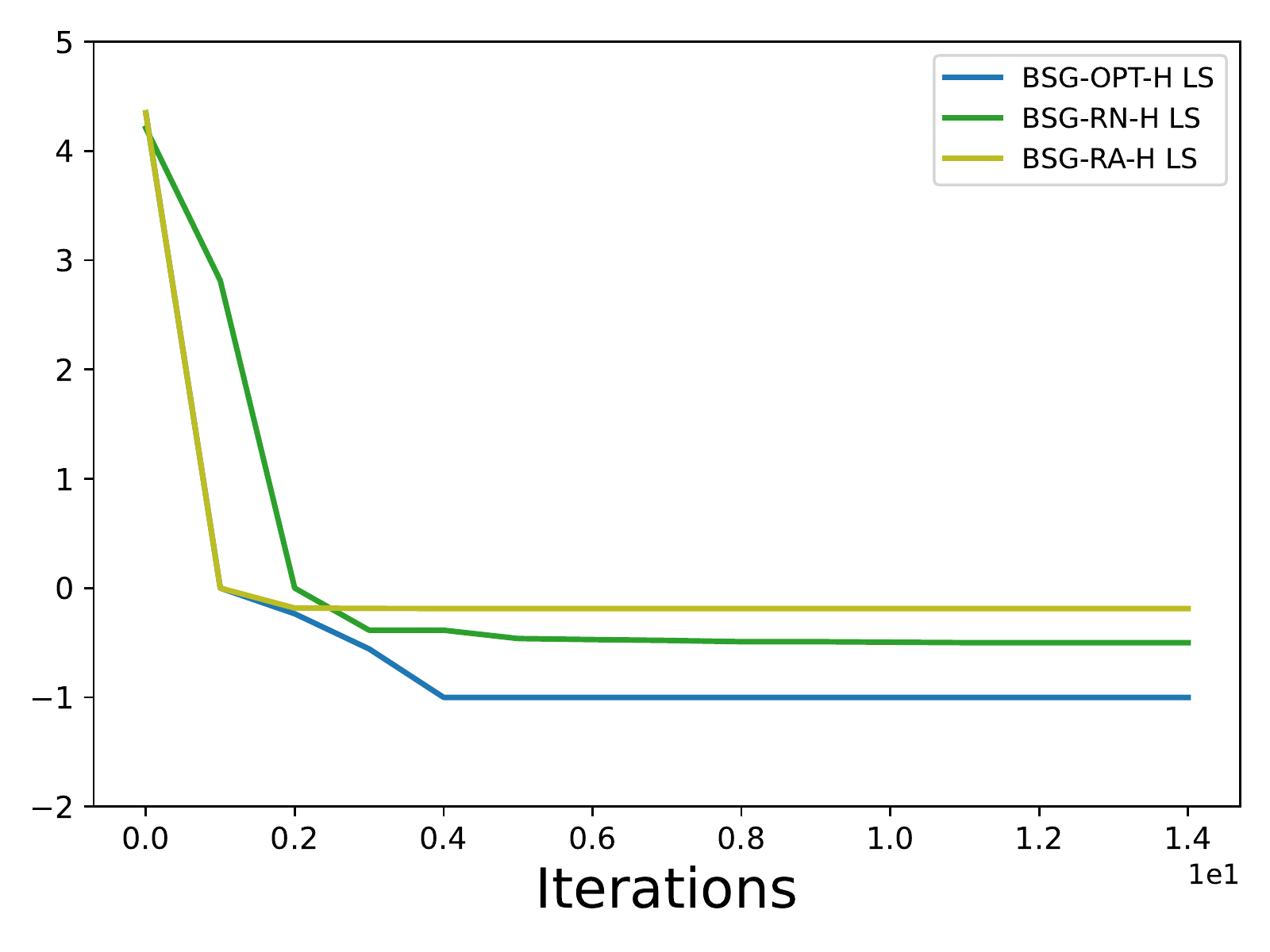}
          \includegraphics[scale=0.38]{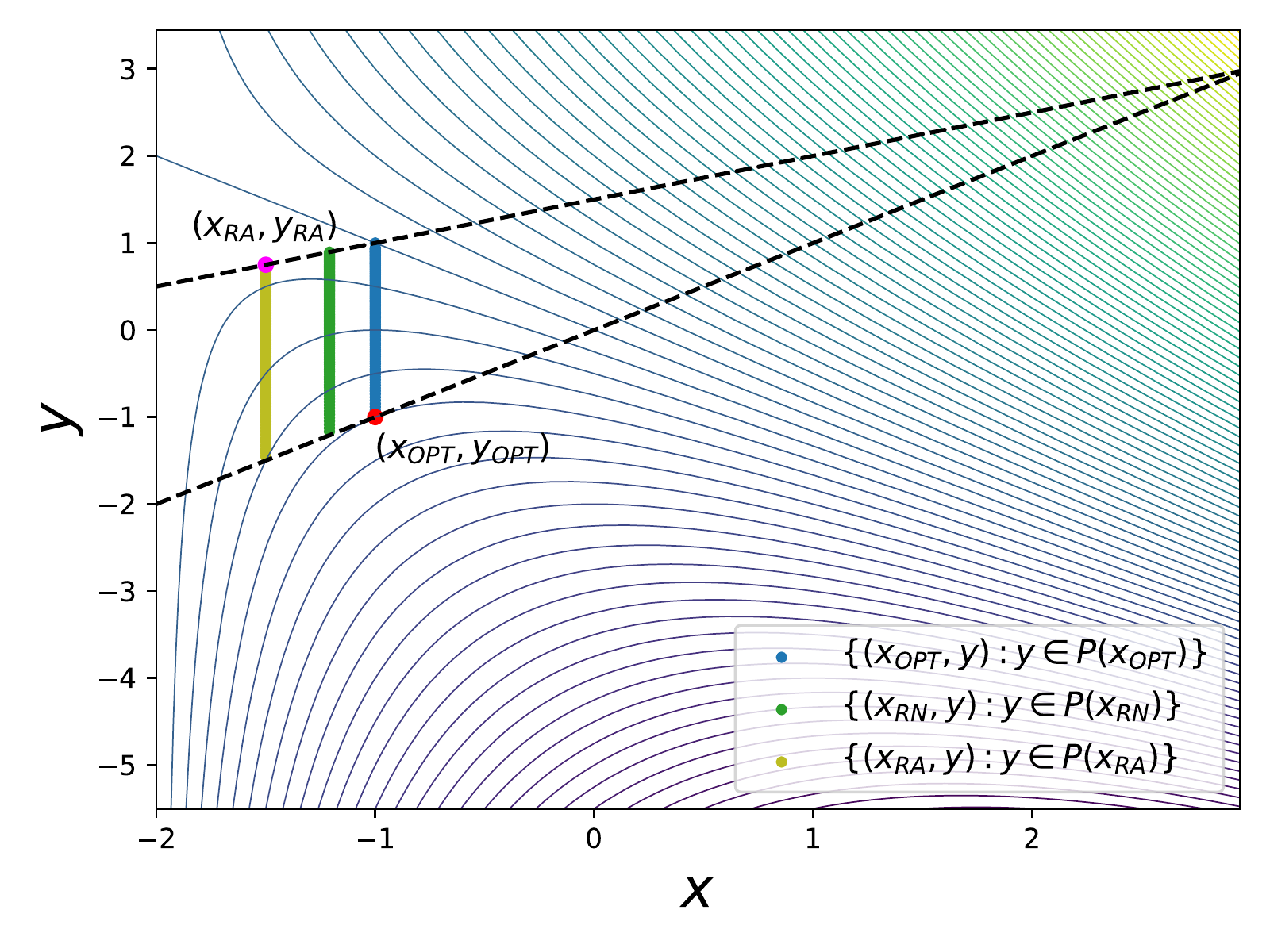} \includegraphics[scale=0.38]{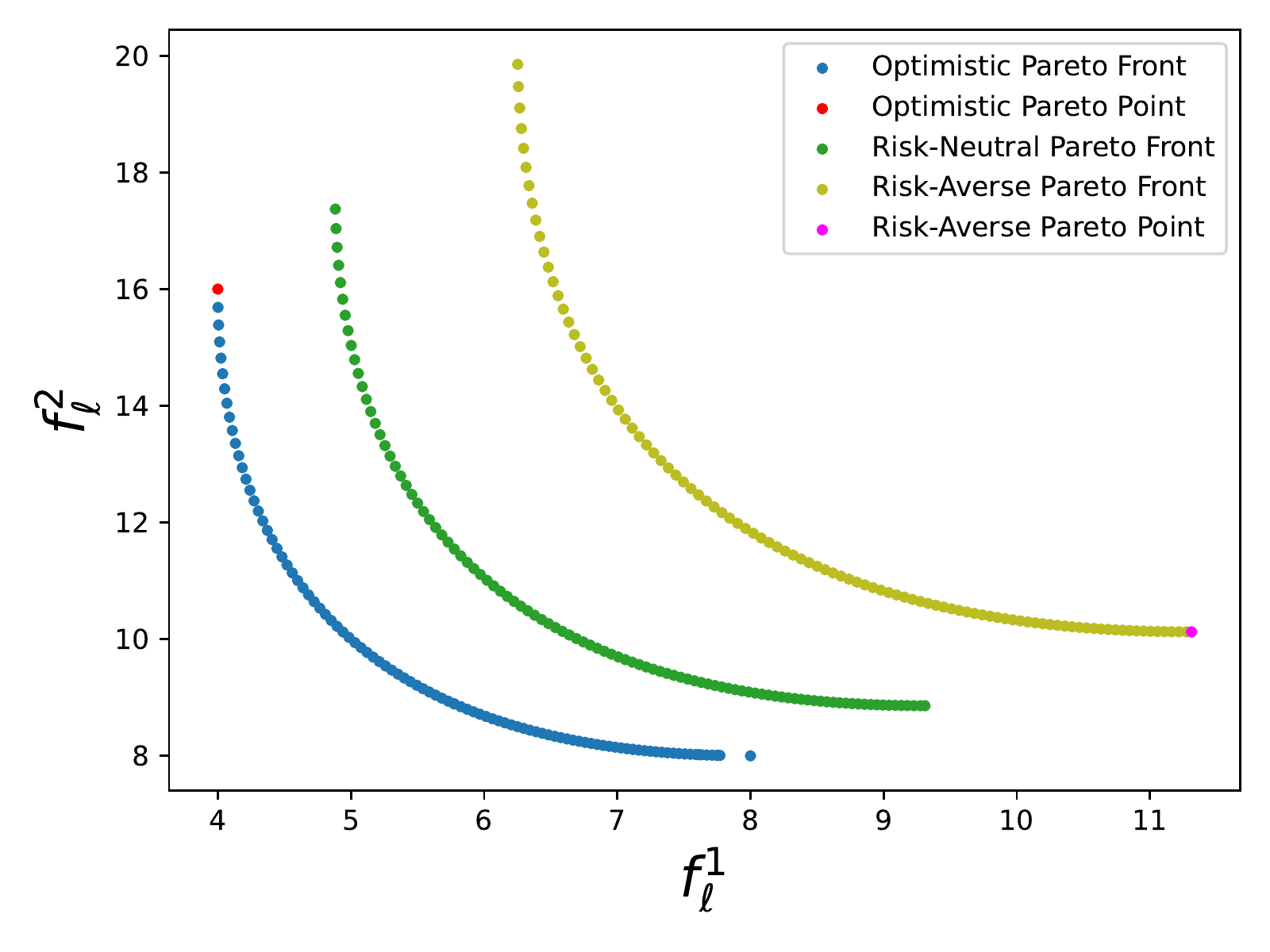} 
        \caption{Results for Problem~2 in the deterministic separable~LL case (for~$\bar{n}$ in Table~\ref{tab:test_prob} equal to~1). In the upper-left plot, the vertical axis represents the values of $f_{\OPT}$, $f_{\RN}$, and $f_{\RA}$.}\label{fig:SP1}
    \end{figure}

    \begin{figure}
    \centering
          \includegraphics[scale=0.38]{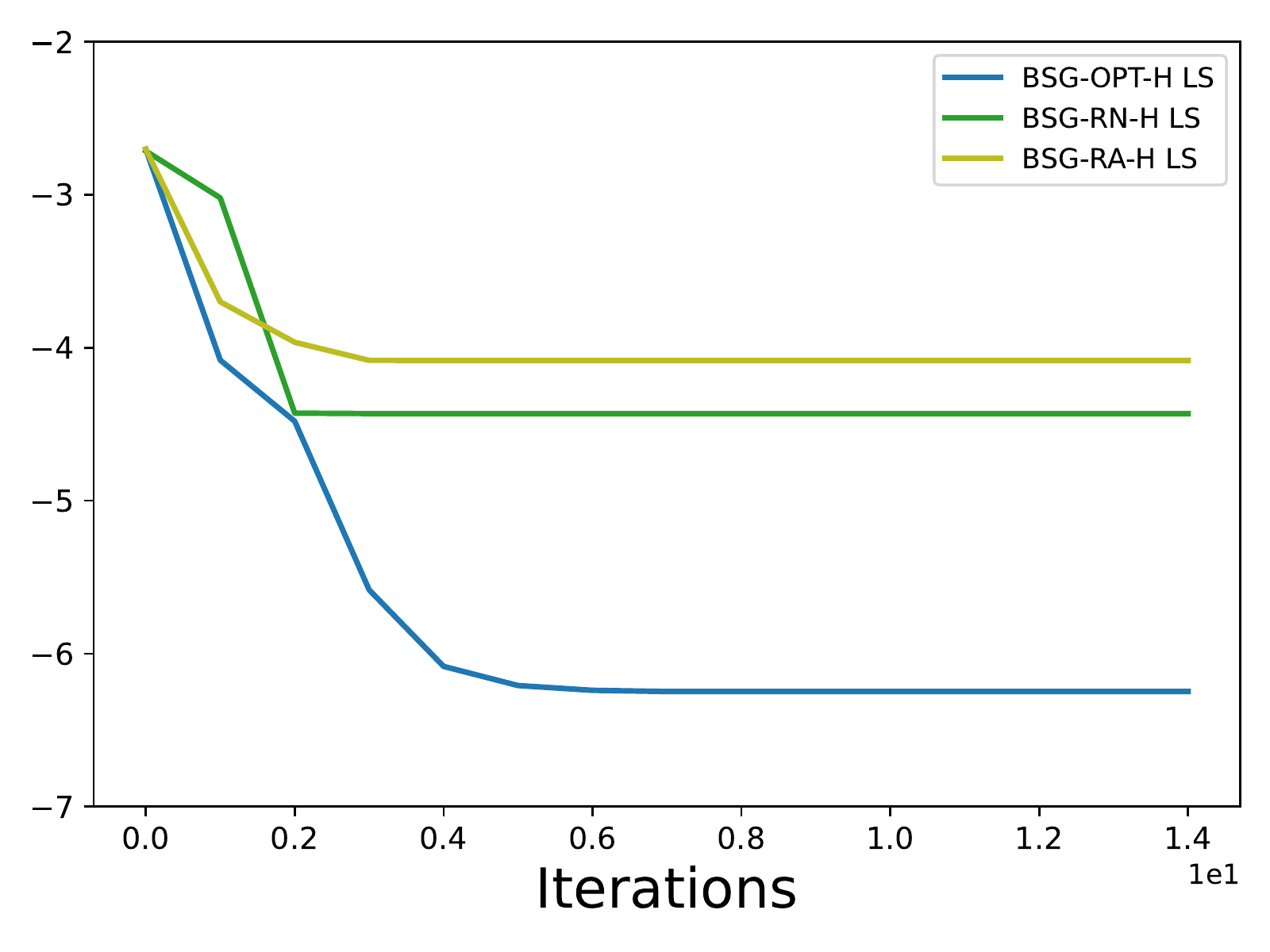}
          \includegraphics[scale=0.38]{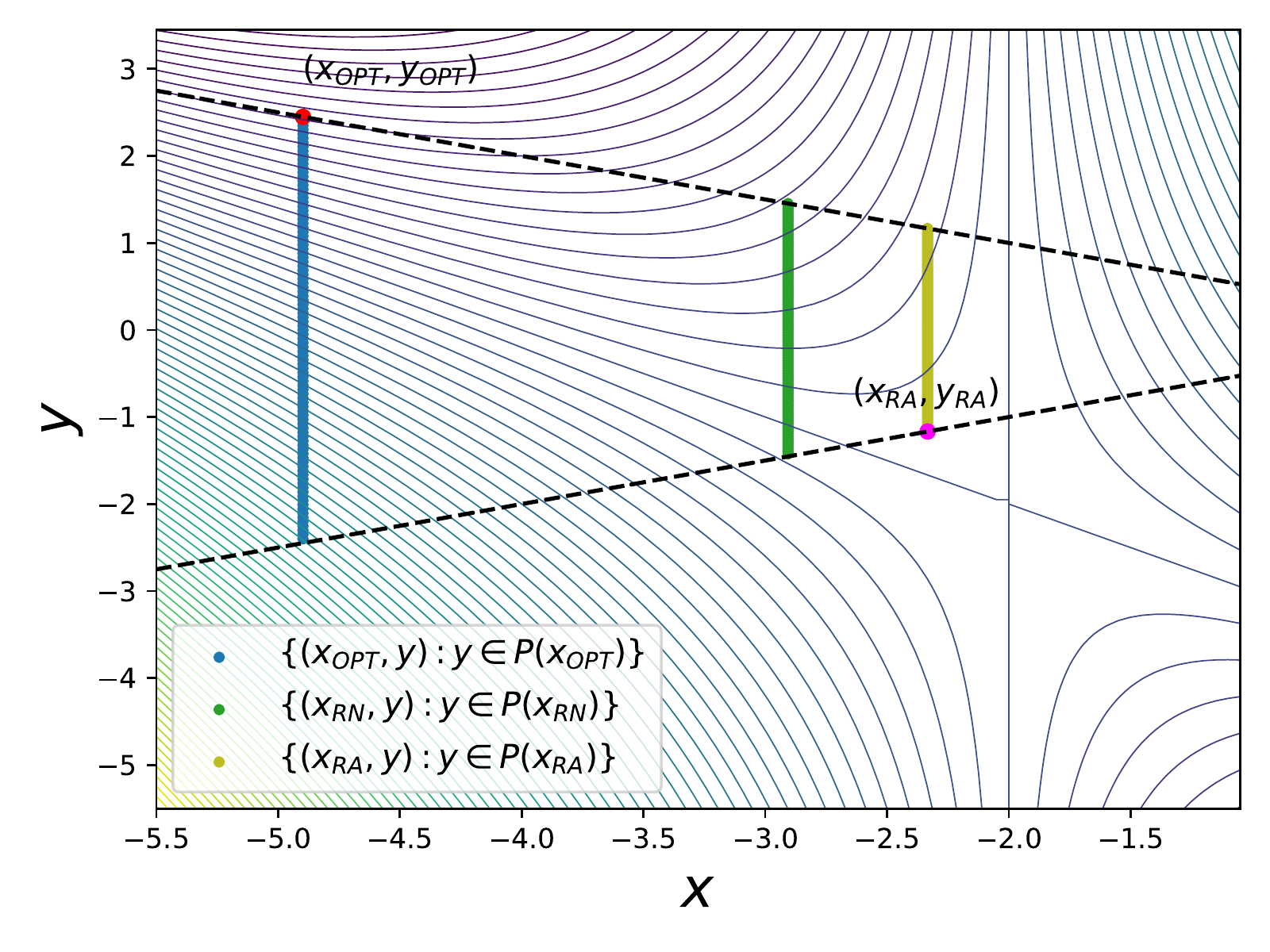} \includegraphics[scale=0.38]{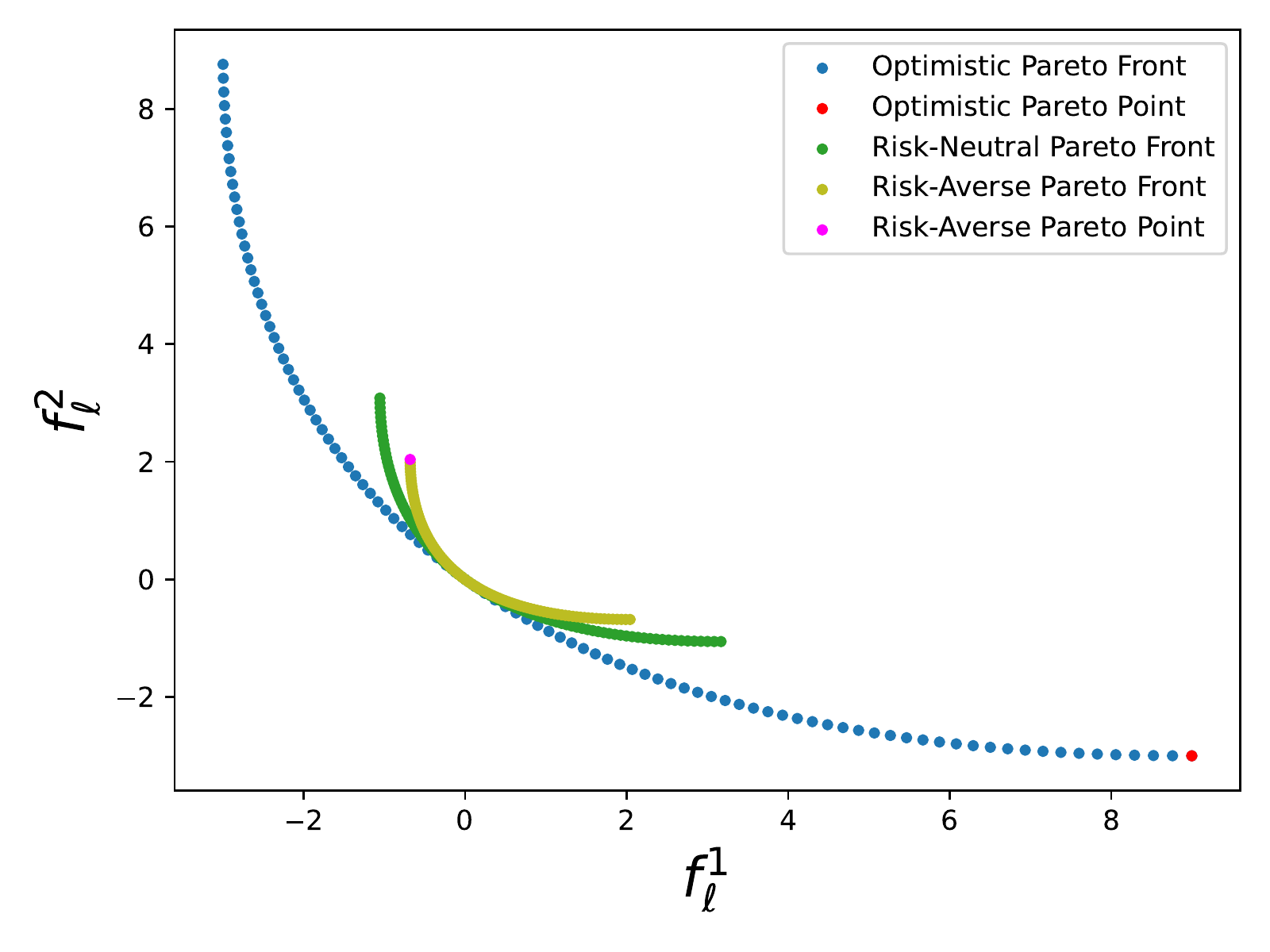} 
        \caption{Results for Problem~3 in the deterministic separable~LL case (for~$\bar{n}$ in Table~\ref{tab:test_prob} equal to~1). In the upper-left plot, the vertical axis represents the values of $f_{\OPT}$, $f_{\RN}$, and $f_{\RA}$.} \label{fig:GKV1}
    \end{figure}

 \paragraph{Deterministic non-separable~LL case.} In this case, we consider Problem~3 from Table~\ref{tab:test_prob} with~$\bar{n} = 50$, $H_3$ and~$H_5$ equal to randomly generated symmetric positive definite matrices, $H_4$ and~$H_6$ equal to identity matrices, and bounds on~$x_i$ given by~$[0,\infty]$. 
 In~\eqref{prob:synthetic_prob}, the components of the vectors~$h_1$ and~$h_2$ have been randomly generated according to a uniform distribution between~$-5$ and~0 and between~$-3$ and~0, respectively. The results obtained are shown in Figure~\ref{fig:GKV1_nonsep}, where the relative positions of the curves are consistent with the ones in Figures~\ref{fig:JOS1}--\ref{fig:GKV1}. Figure~\ref{fig:GKV1_nonsep_BSGRN} shows the results when~BSG-RN is run with~$N = 500$ and~$Q \in \{10,20,40,500\}$ in terms of iterations and time. The results in Figure~\ref{fig:GKV1_nonsep_BSGRN} were also obtained by computing the 95\% confidence intervals produced over 10 randomly generated starting points. Note that the starting point does not seem to have an impact on the convergence of the algorithms here. Further, one can also see that randomly sampling a set of~$Q$ samples from~$\Lambda_{\RN}$ leads to the same optimal function value as using the entire set~$\Lambda_{\RN}$ and, therefore, confirms the validity of the approach. 

    \begin{figure}
    \centering
          \includegraphics[scale=0.4]{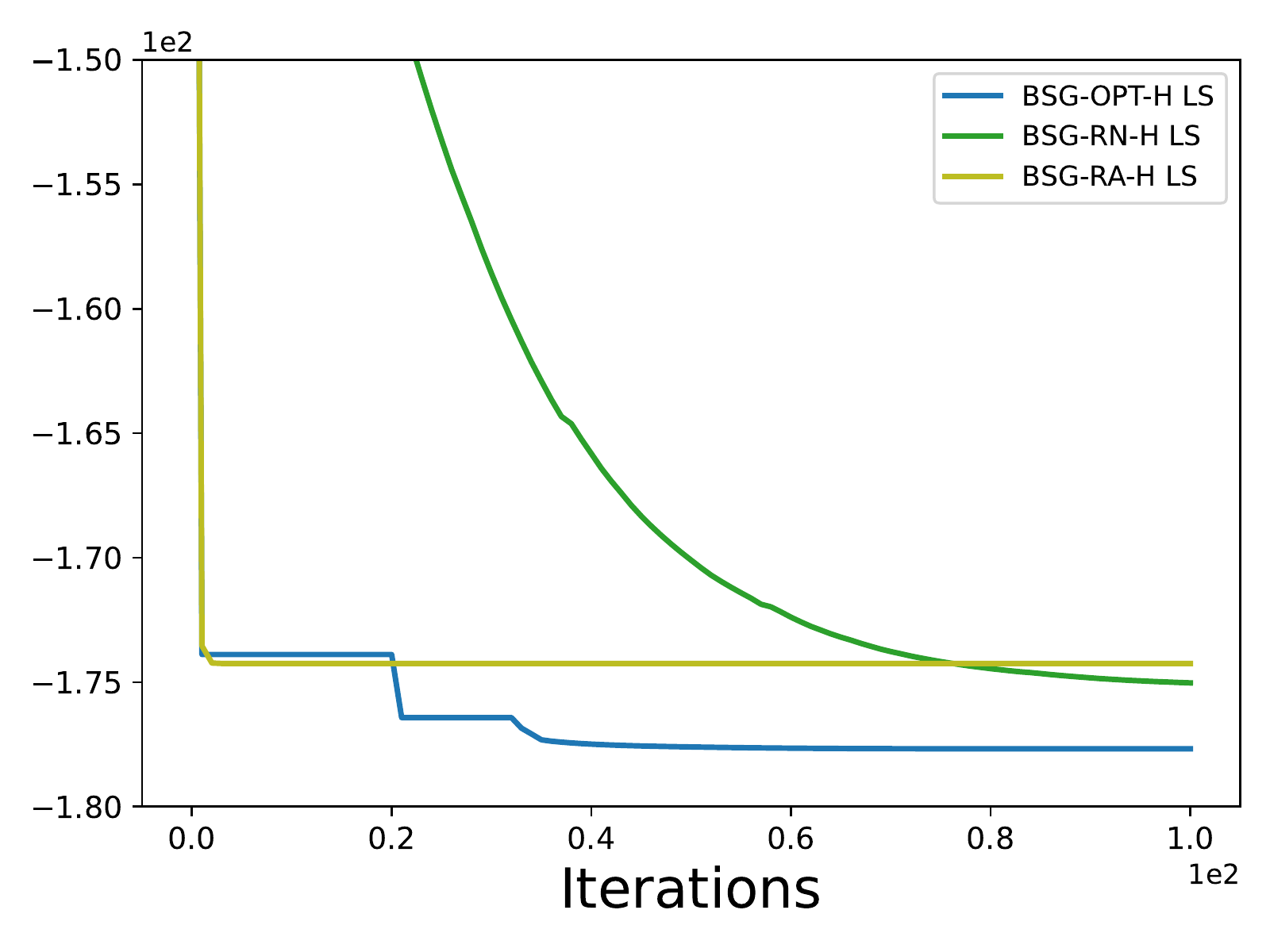} 
          \includegraphics[scale=0.4]{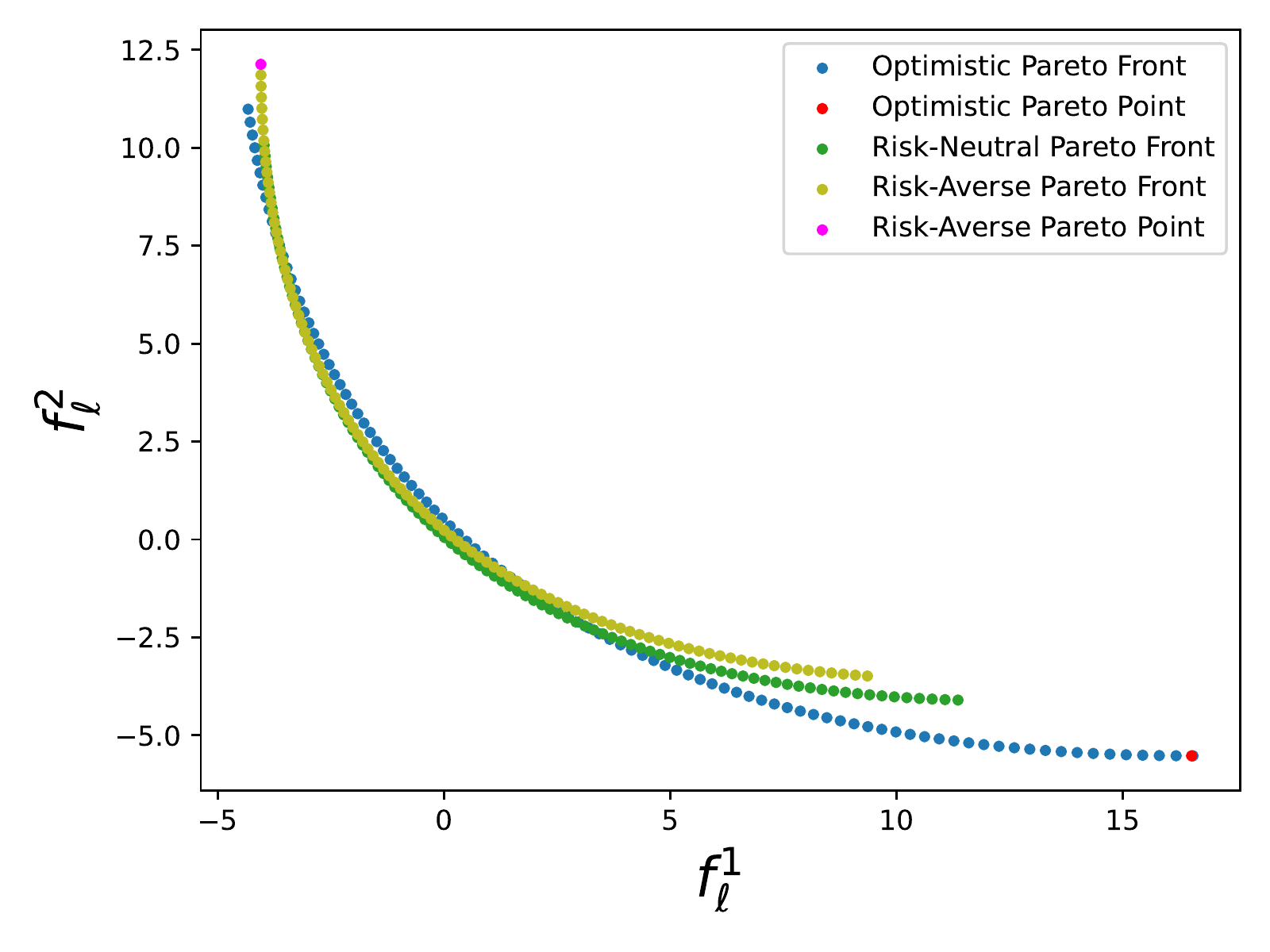} 
        \caption{Results for Problem~3 in the deterministic non-separable~LL case (for~$\bar{n}$ in Table~\ref{tab:test_prob} equal to~50). In the left plot, the vertical axis represents the values of $f_{\OPT}$, $f_{\RN}$, and $f_{\RA}$.}\label{fig:GKV1_nonsep}
    \end{figure}

    \begin{figure}
    \centering
          \includegraphics[scale=0.4]{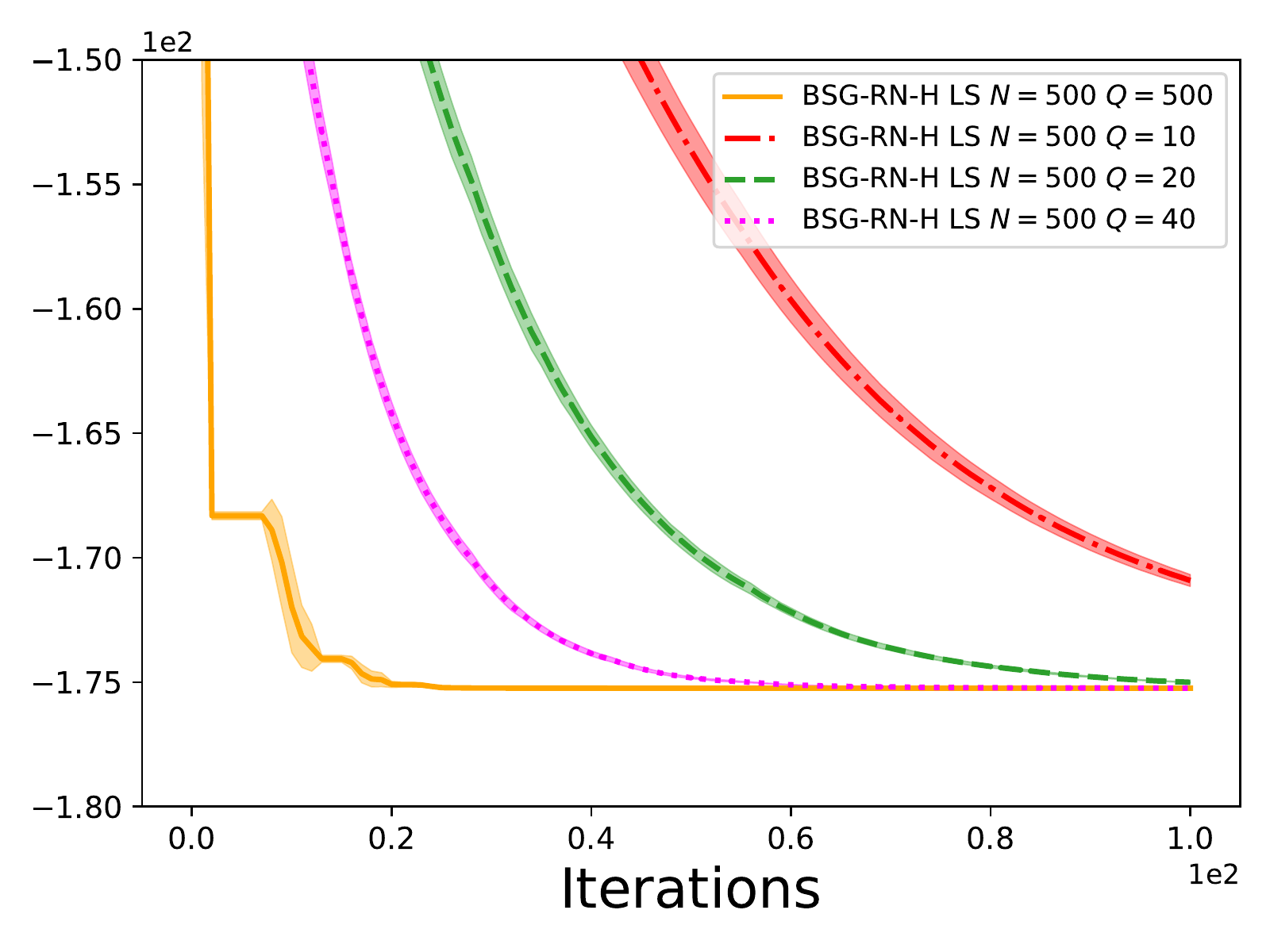} 
          \includegraphics[scale=0.4]{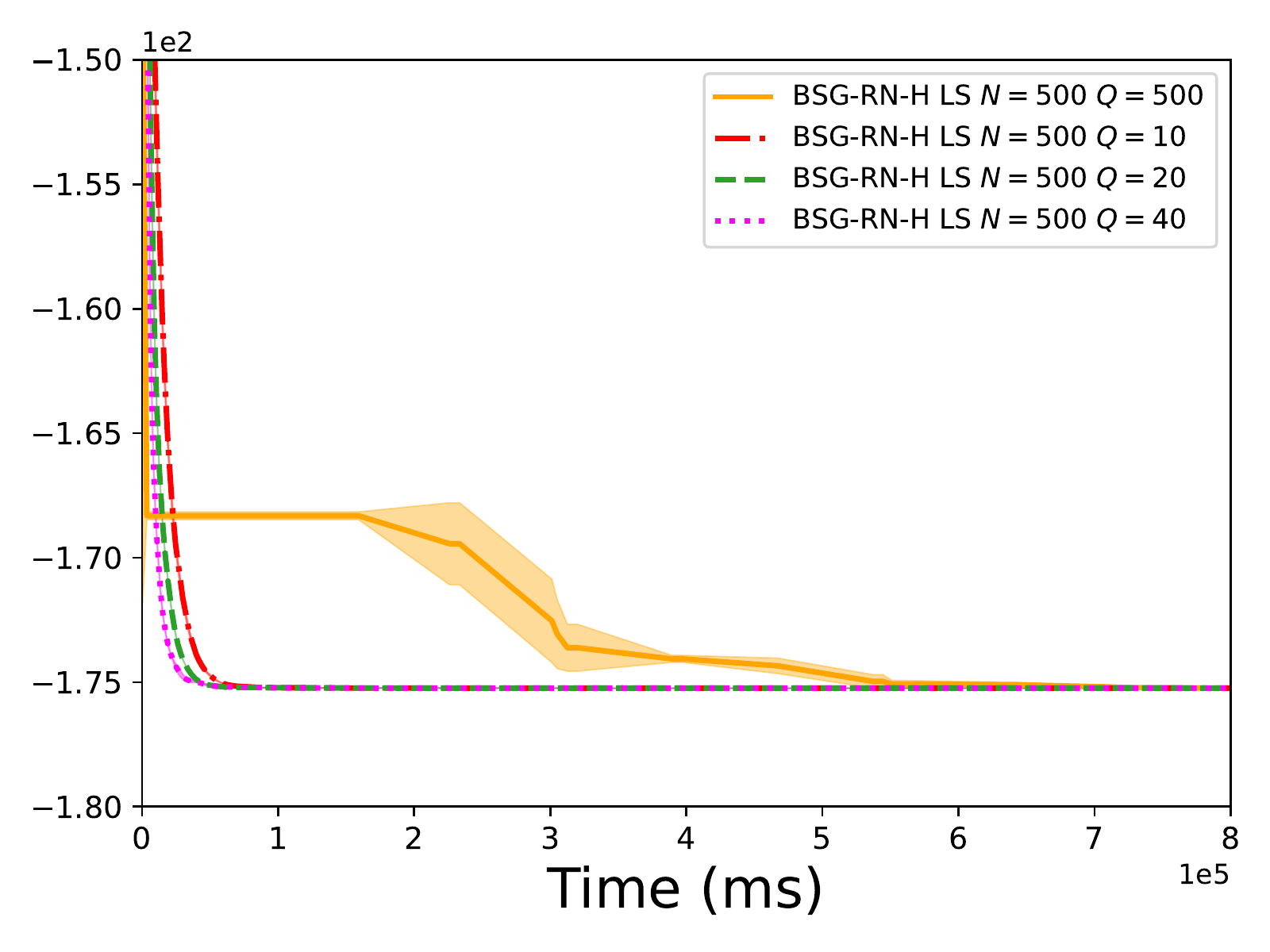} 
        \caption{Results for Problem~3 in the deterministic non-separable~LL case when running~BSG-RN with different values of~$Q$ (for~$\bar{n}$ in Table~\ref{tab:test_prob} equal to~50). In both plots, the vertical axis represents the values of $f_{\OPT}$, $f_{\RN}$, and $f_{\RA}$.}\label{fig:GKV1_nonsep_BSGRN}
    \end{figure}
 
 \paragraph{Stochastic non-separable~LL case.} Note that problem~\eqref{prob:synthetic_prob} is deterministic. To investigate the numerical performance of the stochastic methods considered in the experiments, we compute stochastic gradient and Hessian estimates by adding Gaussian noise with mean~0 to each corresponding deterministic gradient (i.e.,~$\nabla_x f_u$, $\nabla_y f_u$, $\nabla_x f_{\ell}$, $\nabla_y f_{\ell}$) and Hessian (i.e.,~$\nabla^2_{xy} f_{\ell}$, $\nabla^2_{yy} f_{\ell}$). 
 The values of the noise standard deviation were chosen from the set~$\{0,1,2\}$ for the gradients and~$\{0,0.1,0.2\}$ for the Hessians. It is well known that Hessians require larger estimation batches than gradients when considering stochastic settings~\cite[Section 6.1.1]{LBottou_FECurtis_JNocedal_2018}. We compared all the algorithms by using the best~UL and~LL stepsizes found for each of them, which were~$0.1$ for the~UL (1 for BSG-RN) and~$0.001$ for the~LL. Such stepsizes were obtained by performing a grid search over the set~$\{1,0.1,0.01,0.001\}$ for the~UL and the set~$\{0.01,0.001,0.0001\}$ for the~LL. We averaged all the results over 10 trials by using different random seeds and displayed the corresponding 95\% confidence intervals. Figure~\ref{fig:GKV1_nonsep_stoch} shows that as the value of the standard deviation increases, the performance of all the algorithms gets worse. Further, one can also see that~BSG-RN exhibits higher robustness to the noise than~BSG-OPT and~BSG-RA.   

    \begin{figure}
    \centering
          \includegraphics[scale=0.38]{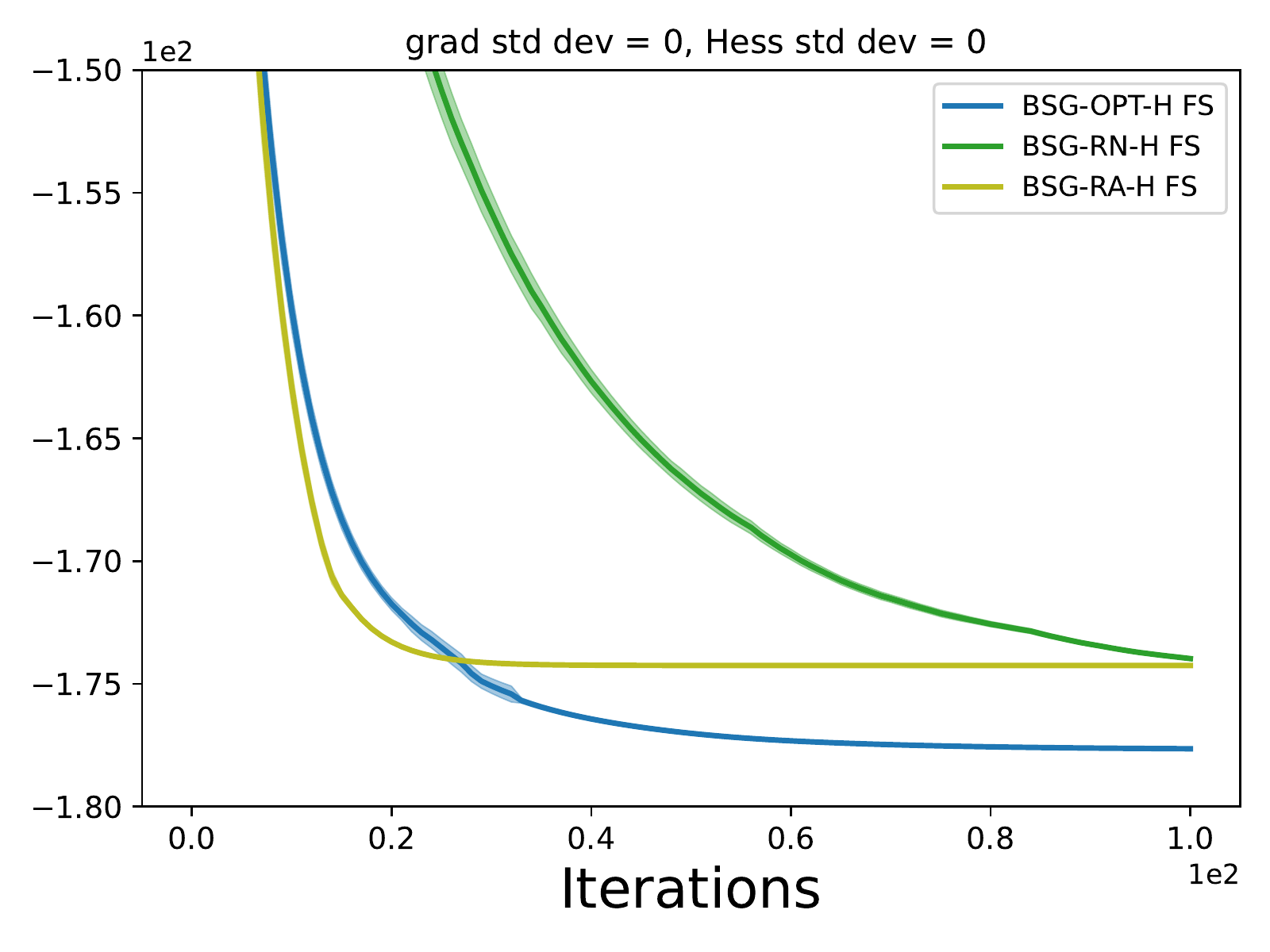} 
          \includegraphics[scale=0.38]{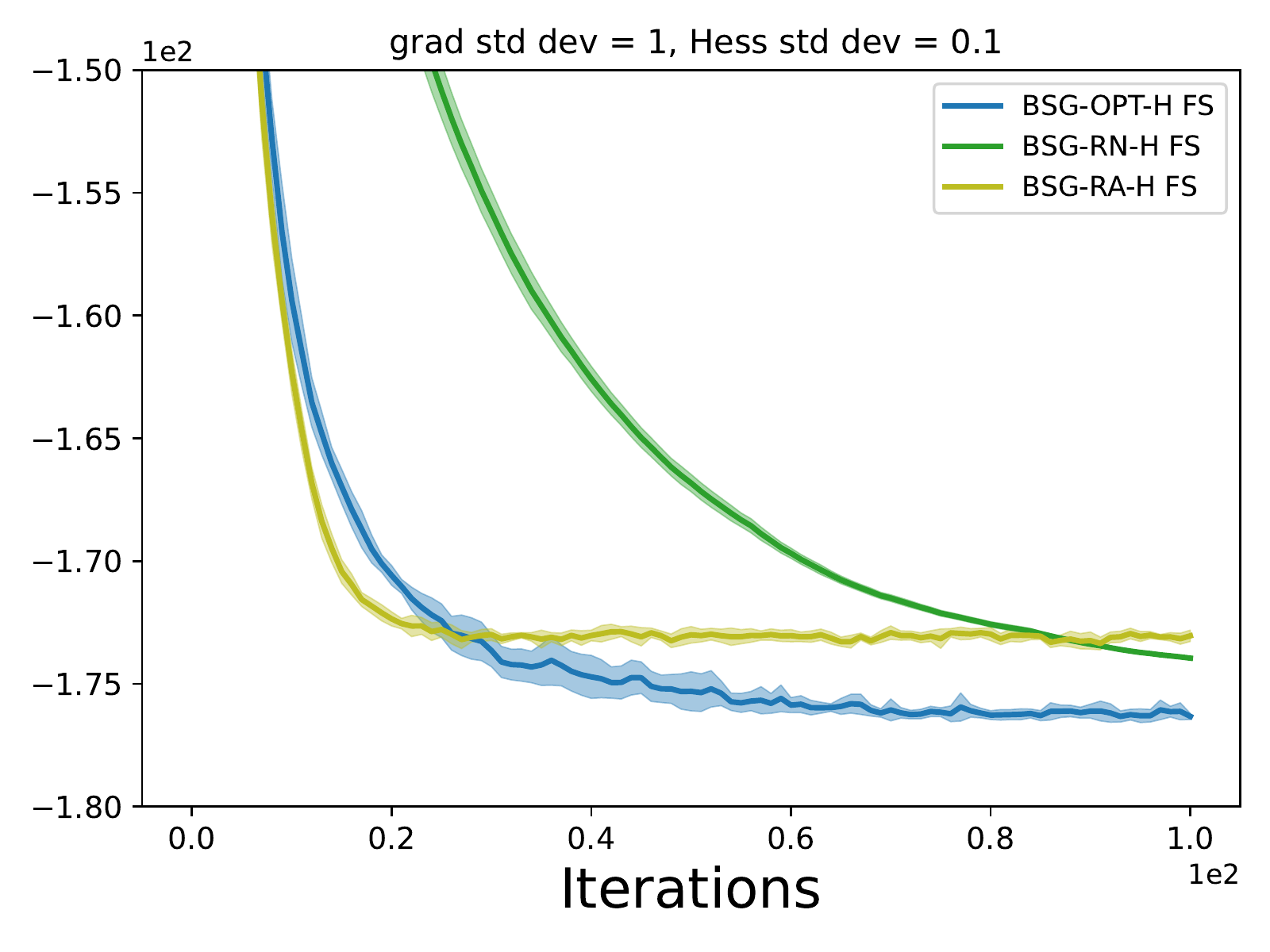} 
          \includegraphics[scale=0.38]{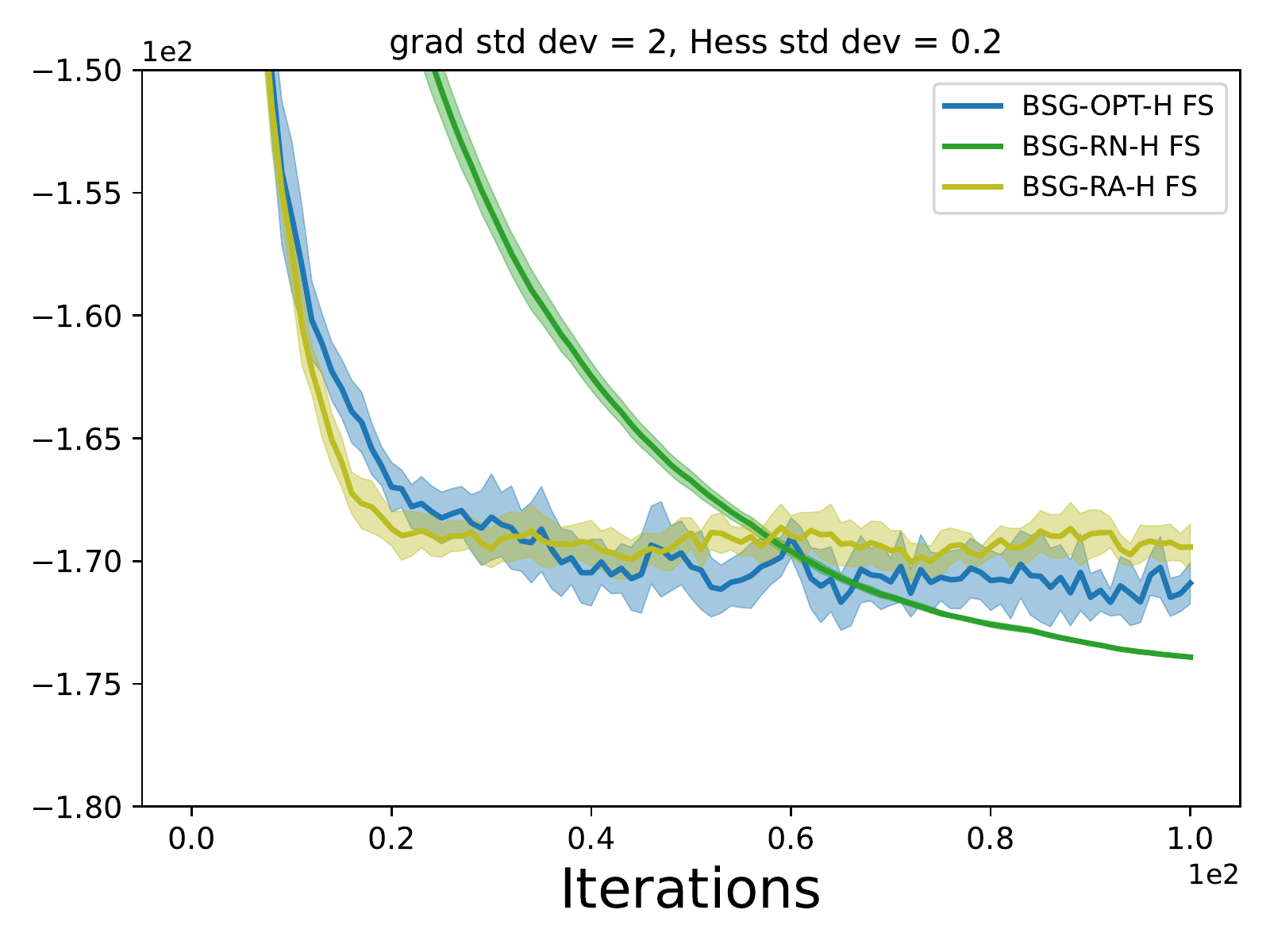}  
        \caption{Results for Problem~3 in the stochastic non-separable~LL case for different standard deviation values of the stochastic gradient and Hessian estimates (for~$\bar{n}$ in Table~\ref{tab:test_prob} equal to~50). The vertical axis represents the values of $f_{\OPT}$, $f_{\RN}$, and $f_{\RA}$.}\label{fig:GKV1_nonsep_stoch}
    \end{figure}

%%%%%%%%%%%%%%%%%%%%%%%%%%%%%%%%%%%%%%%%%%%%%%%%%%%%%%%%%%%%%%%%%%%%%%%%%%%%%%%%%%%%%%%%%
\section{Concluding remarks}\label{sec:UL_multiobj_prob}
%%%%%%%%%%%%%%%%%%%%%%%%%%%%%%%%%%%%%%%%%%%%%%%%%%%%%%%%%%%%%%%%%%%%%%%%%%%%%%%%%%%%%%%%%

In this paper we focused on~BMOLL problems (upper-level single-objective and lower-level multi-objective) for which we developed new risk-neutral and risk-averse formulations for both the deterministic and  stochastic cases and extended the application of the optimistic formulation to the stochastic case. We also developed corresponding gradient-based algorithms.

Other problems that can be considered in bilevel multi-objective optimization are ones with multiple objective functions at the~UL problem or at both the~UL and~LL problems. To address such cases, one can assemble together known approaches from multi-objective and bilevel optimization, as mentioned in Subsections~\ref{sec:adj_grad_UL_multiobj_prob}--\ref{sec:conflicting_UL_LL_obj_funct} below. 

%%%%%%%%%%%%%%%%%%%%%%%%%%%%%%%%%%%%%%%%%%%%%%%%%%%%%%%%%%%%%%%%%%%%%%%%%%%%%%%%%%%%%%%%%
\subsection{The multi-objective single-objective case}\label{sec:adj_grad_UL_multiobj_prob}
%%%%%%%%%%%%%%%%%%%%%%%%%%%%%%%%%%%%%%%%%%%%%%%%%%%%%%%%%%%%%%%%%%%%%%%%%%%%%%%%%%%%%%%%%

To develop a gradient-based algorithm for problem~\eqref{prob:bilevel_multiobj} with~$p \ge 1$ and~$q = 1$, one can consider the equivalent formulation
\begin{equation}\label{prob:bilevel_multiobj_upper}
	\begin{split}
	\min_{x \in \mathbb{R}^n} ~~ & F_u(x,y(x)) \, \text{ s.t. } \, x \in X,
	\end{split}
\end{equation}
where~$y(x)$ denotes the solution of the~LL problem. Note that this is essentially a multi-objective problem where the variables~$y$ are given by~$y(x)$.
Therefore, one can draw inspiration from the multi-gradient method developed for deterministic and stochastic multi-objective optimization, for which one knows how to calculate steepest descent directions (i.e., multi-gradients)~\cite{LGDrummond_BFSvaiter_2005,JFliege_BFSvaiter_2000} and stochastic multi-gradients~\cite{JADesideri_2012,SLiu_LNVicente_2019}. In~(\ref{prob:bilevel_multiobj_upper}), the resulting multi-gradients are given by {\it adjoint multi-gradients}.
Denoting the current iterate as $x$ and the adjoint gradients of the individual~UL objective functions $f^i(x)=f^i_u(x,y(x))$ as $\nabla f^i(x)$, with~$i \in \{1, \ldots, p\}$, the adjoint multi-gradient can be obtained by first solving the QP subproblem
\begin{equation}\label{prob:multigrad}
    \begin{aligned}
     \min_{\delta\in\mathbb{R}^p} ~~ & \left\|\sum_{i=1}^{p} \delta_{i} \nabla f^i(x)\right\|^{2} \quad \text{ s.t. } \quad \delta \in \Delta,
    \end{aligned}
\end{equation}
where~$\Delta = \{\delta \in \mathbb{R}^p : \sum_{i=1}^{p} \delta_i = 1, \delta_i \ge 0 \ \forall i \in \{1, \ldots, p\} \}$ denotes the simplex set (note that~$\Delta$ is a subset of~$\mathbb{R}^p$ and, therefore, is a different simplex set from~$\Lambda$ introduced in~\eqref{eq:simplex_set}, which is a subset of~$\mathbb{R}^q$). Then, the negative adjoint multi-gradient is given by $- \sum_{i=1}^{p} \delta_i^* \nabla f^i(x)$, where $\delta_i^*$ is the optimal solution of problem~\eqref{prob:multigrad}. In the stochastic case, all the gradients and Hessians in the $p$ adjoint gradients used in problem~\eqref{prob:multigrad} are replaced by their corresponding stochastic estimates. 

%%%%%%%%%%%%%%%%%%%%%%%%%%%%%%%%%%%%%%%%%%%%%%%%%%%%%%%%%%%%%%%%%%%%%%%%%%%%%%%%%%%%%%%%%
\subsection{The multi-objective multi-objective case}\label{sec:conflicting_UL_LL_obj_funct}
%%%%%%%%%%%%%%%%%%%%%%%%%%%%%%%%%%%%%%%%%%%%%%%%%%%%%%%%%%%%%%%%%%%%%%%%%%%%%%%%%%%%%%%%%

To address problem~\eqref{prob:bilevel_multiobj} with~$p \ge 1$ and~$q \ge 1$, one can introduce optimistic, risk-neutral, and risk-averse formulations by following similar approaches to the ones described in Sections~\ref{sec:LL_multiobj_prob}--\ref{sec:adj_grad_LL_multiobj_prob_general}. In particular, one can use an adjoint multi-gradient method to solve the~UL problem (see Subsection~\ref{sec:adj_grad_UL_multiobj_prob}) and consider optimistic, risk-neutral, and risk-averse formulations to address the~LL problem in the general case~$P(x)$. In the optimistic and risk-neutral cases, the resulting algorithms differ from each other in terms of the gradients~$\nabla f^i(x)$ used in the~QP subproblem (which is~\eqref{prob:multigrad} in the~LL single-objective case), for all~$i \in \{1,\ldots,p\}$. More specifically, in the optimistic case, the vector of weights~$\lambda \in \mathbb{R}^q$, which is associated with the~LL objective functions, is included among the~UL variables, and the adjoint gradients of the individual objective functions~$f^i_{\OPT}(x,\lambda) = f_u^i(x, y(x,\lambda))$ are given by~$\nabla f^i_{\OPT} = (\nabla_x f^i_{\OPT},\nabla_{\lambda} f^i_{\OPT})$. In the risk-neutral case, each individual objective function is given by~$f^i_{\RN}(x) = (1/N)\sum_{t=1}^{N}f^i_u\left(x,y(x,\lambda^t)\right)$, and the resulting gradients are~$\nabla f^i_{\RN}(x) \; = \; (1/N)\sum_{t=1}^{N} \nabla f^{i\,t}_{\RN}(x)$, where~$f^{i\,t}_{\RN}(x) = f_u^i(x,y(x,\lambda^t))$ for all~$i \in \{1,\ldots,p\}$ and~$t \in \{1,\ldots,N\}$. Developing an algorithm for the risk-averse case by following the approach introduced in Section~\ref{sec:adj_grad_LL_multiobj_prob_general} leads to a constrained problem like~\eqref{prob:single_bilevel_pareto_practical} with multiple objective functions. Solving such a problem requires an algorithm for multi-objective constrained problems and, therefore, further research is needed to make the resulting algorithm efficient, robust, and scalable. 

\section*{Acknowledgments}
This work is partially supported by the U.S. Air Force Office of Scientific Research (AFOSR) award FA9550-23-1-0217.
 
	%\bibliography{../ref-BSG,../ref-idfo,../ref-smg-moo,../ref-bsg-moo,../ref-fairness}

	%\appendix
	
\end{document}